\documentclass[11 pt, twoside]{amsart}
\usepackage[utf8x]{inputenc}
\usepackage{parskip}
\usepackage{amsthm,amsmath,amssymb,oldgerm}
\usepackage[a4paper, margin=2.5cm]{geometry}
\usepackage{verbatim}
\usepackage[all]{xy}
\theoremstyle{plain}
\newtheorem{thm}{Theorem}[section]
\newtheorem{theorem}{Main theorem}

\newtheorem{lem}[thm]{Lemma}

\newtheorem{cor}[thm]{Corollary}
\theoremstyle{definition}
 
\theoremstyle{definition}
\newtheorem{rem}[thm]{Remark}

\newcommand{\mat}[4]{\left(\begin{array}{cc}#1&#2\\#3&#4\end{array}\right)}

\def \A {\mathbb{A}}
\def \Q{\mathbb{Q}}

\def \C {\mathbb{C}}

\def \R {\mathbb{R}}

\def\H {\mathbb H}
\def\Z{\mathbb Z}

\def \calf {\mathcal{F}}

\def \calh {\mathcal{H}}

\def \calo {\mathcal{O}}
\def \cals {\mathcal{S}}

\def\GL{\text{GL}}

\def\Sh{\mathrm{Sh}}

\def\ram{\mathrm{ram}}
\def\spl{\mathrm{spl}}

\DeclareMathOperator{\bk} {{\it{\mathcal{B}_{X}^{\underline{k}}}}}
\DeclareMathOperator{\bkone} {{\it{\mathcal{B}_{X_{\mathcal{G}}}^{k}}}}

\DeclareMathOperator{\hyp}{\mu_{\mathcal{H}}} 
\DeclareMathOperator{\hypnvol}{\mu_{\mathcal{H}^{\it{r}}}^{vol}} 
\DeclareMathOperator{\hypn}{\mu_{\mathcal{H}^{{\it{r}}}}} 
\DeclareMathOperator{\shyp}{\mu_{shyp}^{vol}}

\DeclareMathOperator{\vx}{\mathrm{vol_{\mathrm{hyp}}}}
\DeclareMathOperator{\uk}{{\it{\underline{k}}}}

\DeclareMathOperator{\uv}{{\it{\underline{v}}}}

\title[Estimates of automorphic forms over quaternion algebras]{Estimates of automorphic cusp 
forms over quaternion algebras}
\author{Anilatmaja Aryasomayajula}
\address{Department of Mathematics, Indian Institute of Science Education and Research Tirupati, 
Transit campus at Sri Rama Engineering College, Karkambadi Road,
Mangalam (B.O),Tirupati-517507, India.}
\email{anil.arya@iisertirupati.ac.in}
\author{Baskar Balasubramanyam}
\address{Department of Mathematics, Indian Institute of Science Education and Research Pune,
Dr. Homi Bhabha Road,
Pashan, Pune 411008, India.
}
\email{baskar@iiserpune.ac.in}

\date{\today}
\begin{document}
\begin{abstract}\noindent  
In this article, using methods from geometric analysis and theory of heat kernels, we derive qualitative estimates of automorphic cusp forms defined over quaternion algebras. Using which, we prove an average version of the holomorphic QUE conjecture. We then derive quantitative estimates of classical Hilbert modular cusp forms. This is a generalization of the results from \cite{anil1} and \cite{jk2} to higher dimensions. 

\end{abstract}
\maketitle
The Bergman kernel, which is the reproducing kernel for $L^2$-holomorphic functions defined on a domain in $\mathbb{C}^{n}$ has been extensively studied in complex analysis. The generalization of the Bergman kernel to complex manifolds as the reproducing kernel for the space of global holomorphic sections of a vector bundle carries the information on the algebraic and geometric structures of the underlying manifolds. 

Automorphic forms defined over quaternion algebras are global sections of a holomorphic vector bundle. In this article, we derive asymptotic estimates of the Bergman kernel associated to cuspidal holomorphic vector-valued modular forms of large weight defined over quaternion algebras. We prove an average version of the holomorphic QUE conjecture, when the associated complex analytic space is compact.  We also derive quantitative estimates of the Bergman kernel associated to classical Hilbert modular cusp forms. 

\vspace{0.1cm}\noindent
 In the first half of this article, utilizing certain results from geometric analysis, we derive asymptotic estimates of the Bergman kernel associated to  automorphic cusp forms defined over quaternion algebras. Using which, we prove an average version of the holomorphic QUE conjecture. 

\vspace{0.1cm}\noindent
Jorgenson and Kramer have derived optimal estimates of heat kernels defined over hyperbolic Riemann surfaces of finite volume, and used the estimates to derive optimal 
estimates of the Bergman kernel associated to cusp forms. In the second half of the article, we extend the heat kernel estimates from \cite{jk2} to the setting of Hilbert modular varieties, to derive quantitative estimates of the Bergman kernel associated to classical Hilbert modular cusp forms. Our estimates for the Bergman kernel are optimal, and complement the asymptotic estimates derived in the first half of the article.   

\vspace{0.1cm}\noindent 
Some of the results of this article are known to experts, and the others are very much expected. So this article serves also as a comprehensive collection of estimates for  the Bergman kernel associated to automorphic forms defined over quaternion algebras. 

\bigskip
\section{Introduction}\label{introduction}
In this section, we state and explain our main results, and also similar results from literature. Before introducing our results, we first set up our notation. 
\subsection{Automorphic cusp forms defined over a quaternion algebra}\label{notation1}
Let $F$ be a totally real number field of degree $d$. Let $\calo_F$ denote the ring of integers of $F$. Let $B$ be a quaternion algebra defined over $F$, and let $S$ be the finite set of places of $F$, where $B$ is ramified. Let $\Sigma$ denote the set of infinite places of $F$ and let $\Sigma_\ram$ denote the set of infinite places where $B$ is ramified.  Let $\Sigma_\spl:=\Sigma\backslash\Sigma_\ram$ denote the infinite places where the quaternion algebra $B$ is unramified. We assume  that $r:=|\Sigma_\spl|\not=0$.  

Let $\A$ denote the adele ring over $\Q$, and let $\A_f$ denote the finite part of the adeles. Let $G$ be the restriction of scalars from $F$ to $\Q$ of the algebraic group associated to the units in $B$. Let $Z$ denote the center of $G$. Denote by $G_\infty = G(\R)$, the real points of $G$. Similarly, let $Z_\infty = Z(\R)$. Let $K_\infty$ denote the maximal compact subgroup of $G_\infty$. For any open compact subgroup $U \subset G(\A_f)$, the Shimura variety attached to $B$ of level $U$ is given by 
\begin{align*}
\Sh_U = \Sh^G_U = G(\Q) \backslash G(\A) /U (Z_\infty K_\infty)^\circ,
\end{align*}
where ${}^\circ $ denotes the connected component containing  the identity element.

Let $\mathcal{H}$ denote the hyperbolic upper half-plane. The Shimura variety $\Sh_U$ admits the following decomposition into connected components
\begin{align}\label{shdecomp}
\Sh_U = \bigsqcup_{i=1}^h \Gamma_i \backslash \calh^r,
\end{align}
where $\Gamma_i$ are certain discrete subgroups of $\prod_{v \in \Sigma_{\spl}} \GL_2 (\R)^+$, and $\GL_2(\R)^+$ denotes the matrices with positive determinant.  We take $U$ sufficiently small so that the Shimura variety $\Sh_U$ has no elliptic points. We refer the reader to section \ref{secauto} for more details.

Let $\Gamma$ be any of the $\Gamma_{i}$ from equation \eqref{shdecomp}, and let $X$ denote 
the Shimura variety $\Gamma\backslash \mathcal{H}^{r}$. We identify $X$ locally with its universal cover $\mathcal{H}^{r}$, since $X$ has no elliptic points. Let $\hyp$ denote the hyperbolic metric on $\mathcal{H}$, which for any $z=x+iy\in\mathcal{H}$ is given by the following formula
\begin{align*}
 \hyp(z):=\frac{i}{2}\cdot\frac{dz\wedge d\overline{z}}{y^{2}}=\frac{dx\wedge dy}{y^{2}}.
\end{align*}
Then, for $(z_1,\dots z_r)\in\mathcal{H}^{r}$, the hyperbolic metric on $\mathcal{H}^{r}$ is given by the following formula
\begin{align}\label{defnhyp}
 \hypn(z)=\sum_{i=1}^{r}\hyp(z_i),
\end{align}
which defines a K\"ahler metric on the Shimura variety $X$, compatible with the complex 
structure of $X$. Let $\hypnvol$ denote the volume form associated to the hyperbolic metric 
$\hypn$, and let $\vx(X)$ denote the volume of $X$ with respect to $\hypnvol$. Let $\shyp$ denote the rescaled hyperbolic volume form which gives $X$ volume $1$. 

When $B = M_2 (F)$, the Shimura variety $X=\Gamma\backslash\mathcal{H}^{d}$ is a noncompact complex K\"ahler manifold of finite hyperbolic volume $\vx(X)$ with cusps, and is otherwise a compact K\"ahler manifold. 

Let $\mathcal{S}_{\underline{k},\underline{v}}(\Gamma)$ denote the complex vector space of automorphic cusp forms of weight-$\uk$ with respect to $\Gamma$, where $\uk:=(\uk_\spl,\uk_\ram)$ and $\uk_\spl, \uk_\ram \in \mathbb{Z}_{>0}^{r},  \mathbb{Z}_{>0}^{d-r}$, respectively. The vector $\underline{v}\in \mathbb{Z}_{>0}^{d}$ depends on the vector $\uk$ and a choice of parallel defect, and we refer the reader to section \ref{secauto}  for further details regarding  the definition of  automorphic forms. 
\begin{rem}
We only work with automorphic forms of weight-$\uk=(\uk_\spl,\uk_\ram)$, where $\uk_\spl$ is of the form $(k,\dots,k)\in  \mathbb{Z}_{>0}^{r}$, and for the rest of the article we fix the vector $\uk_\ram \in\mathbb{Z}_{>0}^{d-r}$.  When $X$ is noncompact, $d=r$ and $\uk=(k,\dots,k)\in\mathbb{Z}_{>0}^{d}$, i.e., we have no vector $\uk_\ram$. Even as the weight-$\uk$ varies, we will want it to satisfy a parity condition which is defined in section \ref{secauto}. 
\end{rem}
For $f\in \mathcal{S}_{\underline{k},\underline{v}}(\Gamma)$, the Petersson metric at the point $z=(z_{1}=x_{1}+iy_{1},\dots, z_{r}=x_{r}+iy_{r})\in X$ is given by 
\begin{align}\label{petmetric}
\|f\|_{\mathrm{pet}}^{2}(z):=\bigg(\prod_{i=1}^{r}y^{k}_{i}\bigg)\big(f(z)\cdot {^t}\overline{f(z)}\big).
\end{align}
Let $\delta_{\uk}$ denote the dimension of the complex vector space $\mathcal{S}_{\underline{k},\underline{v}}(\Gamma)$. Let $\lbrace f_{1},\ldots,f_{\delta_{\uk}}\rbrace$ denote an orthonormal basis of $\mathcal{S}_{\underline{k},\underline{v}}(\Gamma)$ with respect to the Petersson inner-product. For $z\in X$, the Bergman kernel associated to the space $\mathcal{S}_{\underline{k},\underline{v}}(\Gamma)$ is given by the following formula 
\begin{align}\label{bcuspdefn}
\bk(z):=\sum_{i=1}^{\delta_{\uk}}\|f_{i}\|_{\mathrm{pet}}^{2}(z).
\end{align}
There exists a finite Galois cover $\pi:X^{1} \longrightarrow X$ of $X$, where $X^{1}:=\Gamma^{1}\backslash\mathcal{H}^{r}$, and we refer the reader to section 
\ref{secauto} for further details regarding the group $\Gamma^1$. Let $[X^1 : X]$ denote the degree of the cover $\pi:X^{1} \longrightarrow X$.

Let $\mathcal{W}_{\uk}$ be the automorphic vector bundle of weight-$\underline{k}=(k,\dots,k,\uk_{\ram})\in \mathbb{Z}_{\ge 2}^{r} \times \mathbb{Z}_{\ge 2}^{d-r}$. The space of holomorphic global sections of the vector bundle $\mathcal{W}_{\uk}$ denoted by $H^{0}\big(X,\mathcal{W}_{\uk}\big)$ is in fact equal to $\mathcal{S}_{\underline{k},\underline{v}}(\Gamma)$. For any $\uk$, let $\mathrm{rank} (\mathcal{W}_{\uk})$ denote the rank  of the vector bundle $\mathcal{W}_{\uk}$. This rank depends only on $\uk_\ram$ and since we have fixed this quantity, the rank remains a constant in all our considerations.

Let $k_{\circ}$ be the smallest positive integer such that the weight-$\underline{k}_{\circ}:=(k_{\circ},\dots,k_{\circ},\uk_{\ram})\in\mathbb{Z}_{\ge 2}^{r}\times\mathbb{Z}_{\ge 2}^{d-r}$ satisfies the parity condition. The weights we are interested are of the form $\uk = (k, \dots, k, \uk_\ram)$ as $k \in k_\circ \Z_{>0}$ goes to infinity. 
With notation as above, let the quaternion algebra $B$ be split at all finite places of $F$, and let the vector $k_{\mathrm{ram}}$ be trivial. For $\uk:=(k_1,\ldots,k_r) \in(2\mathbb{Z}_{ >0})^r$, let  $\cals_{\uk}(\Gamma)$ denote the complex vector space of weight $\uk$- automorphic cusp forms, which are also known as adelic Hilbert modular cusp forms. The automorphic cusp forms associated to the finite subgroup $\Gamma^1$ of $\Gamma$ are the classical Hilbert modular cusp forms.  In the case of classical Hilbert modular cusp forms, we are not confined to parallel weights alone. We assume that $\Gamma^1$ is one of the following two types:

\vspace{0.1cm}
{\it{Type (1)}}: Let $1\leq r<d$, which implies that $\Gamma^1$ is a discrete, irreducible, cocompact subgroup of $\mathrm{PSL}_{2}(\mathbb{R})^{r}$, and the Shimura variety $X^1=\Gamma^{1}\backslash\calh^r$ is a compact complex manifold of dimension $r$, and the hyperbolic metric $\mu_{\mathcal{H}^r}$ defines a K\"ahler metric on $X^{1}$.  

\vspace{0.15cm}
{\it{Type (2)}}: Let the quaternion algebra be split everywhere, i.e., $B=M_{2}(F)$, which implies that $\Gamma^1$ is a finite index subgroup of $\mathrm{PSL}_{2}(\calo_{F})$ without elliptic fixed points. This implies that $X^{1}=\Gamma^{1}\backslash\calh^d$ admits the structure of a noncompact complex manifold of dimension $d$ with cusps. The hyperbolic metric $\mu_{\mathcal{H}^d}$, which is the natural metric on $\calh^d$ defines a K\"ahler metric on $X^1$. 

In both the above cases,  we assume that $\Gamma^1$ does not contain hyperbolic-elliptic elements. This assumption is only to ensure the brevity of of exposition, and the cases of both elliptic elements and hyperbolic-elliptic elements can be easily tackled. 

As above, let $\vx(X^{1})$ denote the volume of $X^{1}$ with respect to the hyperbolic volume form $\mu^{\mathrm{vol}}_{\mathcal{H}^r}(z)$. 
Let $\shyp(z)$ denote the rescaled hyperbolic metric $\mu^{\mathrm{vol}}_{\mathcal{H}^r}(z)\slash \vx(X^1)$, which gives $X^1$ volume $1$.

For $\uk:=(k_1,\ldots,k_r) \in(2\mathbb{Z}_{> 0})^r$, let $\mathcal{S}_{\uk}(\Gamma^1)$ denote the complex vector space of weight-$\uk$ cusp forms with respect to $\Gamma^1$. The point-wise Petersson metric on $\mathcal{S}_{\uk}(\Gamma^1)$ is as in equation \eqref{petmetric}. Let $\delta_{\uk}$ denote the dimension of $\mathcal{S}_{\uk}(\Gamma^1)$, and let $\lbrace f_{1},\ldots,f_{\delta_{\uk}} \rbrace$ denote an orthonormal basis of $\mathcal{S}_{\uk}(\Gamma^1)$ with respect to the Petersson inner-product. For any $z=(x_1+iy_1,\ldots,x_r+iy_r)\in\calh^r$, put 
\begin{align}\label{defnbergmanclassical}
\mathcal{B}_{X^1}^{\uk}(z):=\sum_{i=1}^{\delta_{\uk}}\bigg(\prod_{j=1}^{r}y_j^{k_j}|f_{i}(z)|^{2}\bigg).
\end{align} 
\subsection{Statements of main results}\label{statementofresults}
To motivate our main results, we now recall relevant results associated to Fuchsian subgroups of $\mathrm{PSL}_{2}(\mathbb{R})$. For the rest of this section, let $N\in\mathbb{Z}_{> 0}$ with $N$ square-free.

Let $f$ be any Hecke-normalized newform of $\Gamma_{0}(N)$ with trivial nebentypus and of weight-$2$. 
Then, in \cite{abbes}, Abbes and Ullmo proved the following estimate 
\begin{align}\label{abbesestimate}
\sup_{z\in\mathcal{H}} y|f(z)|=O_{\varepsilon}(N^{\frac{1}{2}+\varepsilon}),
\end{align}
for any $\varepsilon > 0$, and the implied constant depends on $\varepsilon$. 

Using heat kernel techniques, in \cite{jk1}, Jorgenson, and Kramer have re-proved the result of Abbes and Ullmo (i.e., estimate \eqref{abbesestimate}). In \cite{jk2}, Friedman, Jorgenson and Kramer extended their method from \cite{jk1}, to derive sup-norm bounds for $\bkone(z)$. Here,  as in \eqref{bcuspdefn}, $\bkone(z)$ is the Bergman kernel associated to $S_{k}(\mathcal{G})$, the space of weight-$k$ cusp forms associated to $\mathcal{G}$, where $\mathcal{G}$ is a cofinite Fuchsian subgroup of $\mathrm{PSL}_{2}(\mathbb{R})$,  and $k\in2\mathbb{Z}_{>0}$. When $X_{\mathcal{G}}:=\mathcal{G}\backslash\mathcal{H}$ is a compact hyperbolic Riemann surface, they proved the following estimate
\begin{align}\label{jkestimate1}
 \sup_{z\in X_{\mathcal{G}}}\bkone(z)=O(k);
\end{align}
and when $X_{\mathcal{G}}$ is a noncompact hyperbolic Riemann surface of finite volume, they proved the following estimate
\begin{align}\label{jkestimate2}
 \sup_{z\in X_{\mathcal{G}}}\bkone(z)=O(k^{3\slash 2}),
\end{align}
where the implied constants in both the above estimates are independent of $X$. The estimates of Jorgenson and Kramer are optimal, as shown in \cite{jk2}. 

For $k\in\mathbb{R}_{>0}$ with $k>2$, Bergman kernel can be represented by an infinite series, which is uniformly convergent in $z\in X$. Using which, Steiner has extended the bounds of Jorgenson and Kramer to real weights.  Let $\mathcal{G}$ be any subgroup of finite index in $\mathrm{SL}_{2}(\mathbb{Z})$, let $k\in\mathbb{R}_{>0}$ with $k\gg1$, and let $\nu$ denote the factor of automorphy of weight $k$ with the associated character being unitary. Furthermore, let $A$ be a compact subset of $X_{\mathcal{G}}$, and let $\mathcal{B}^{k,\nu}_{X_{\mathcal{G}}}(z)$ denote the Bergman kernel associated to the complex vector space of weight-$k$ cusp forms . Then, in \cite{raphael}, Steiner has derived the following estimates
\begin{align*}
 \sup_{z\in A}\mathcal{B}^{k,\nu}_{X_{\mathcal{G}}}(z)=O_{A}(k),
\end{align*}
where the implied constant depends on the compact subset $A$; and 
\begin{align*}
 \sup_{z\in X_{\mathcal{G}} }\mathcal{B}^{k,\nu}_{X_{\mathcal{G}}}(z)=O_{X}\big(k^{3\slash 2}\big),
\end{align*}
where the implied constant depends on $X$. 

In \cite{anil1}, using similar techniques as in this article, when $X_{\mathcal{G}}$ is a compact hyperbolic Riemann surface, for $z\in X_{\mathcal{G}}$, the first named author has  shown that
\begin{align}\label{estimate1}
\lim_{k\rightarrow \infty}\frac{1}{k}\bkone(z)=\frac{1}{4\pi};
\end{align}
and when $X_{\mathcal{G}}$ is a noncompact hyperbolic Riemann surface of finite volume, the first named author has  shown that
\begin{align}\label{estimate2}
\limsup_{k\rightarrow \infty}\frac{1}{k}\bkone(z)\leq\frac{1}{4\pi}.
\end{align}
Equation \eqref{estimate2} complements estimate \eqref{jkestimate2} in the sense that, 
when $X_{\mathcal{G}}$ is noncompact, for a given $k\in\mathbb{Z}_{>0}$, excepting for finitely many values of $z\in X_{\mathcal{G}}$, the function $\bkone(z)$ satisfies the  estimate $O(k)$. In fact, equations \eqref{estimate1} and \eqref{estimate2} have been extended to half-integral weights and for nontrivial nebentypus in the same article. 

We now state our main results. We have three main results, one is regarding the  asymptotic 
estimates of automorphic cusp forms, and the second one is a proof of the average version of the 
QUE conjecture.  Lastly, the third main result is about the quantitative estimates of classical Hilbert modular cusp forms. 

We now state the first main result, which is proved as Theorem \ref{thm1} and Corollary \ref{cor2} in section \ref{results1}. %
\begin{theorem}\label{mainthm1}
Let the notation be as in section \ref{notation1}. If the Shimura variety $X$ is compact, for $k\in k_{\circ} \mathbb{Z}_{>0}$, $\uk = (k, \dots, k, \uk_\ram)$ and $z\in X$, we have the following equality
\begin{align*}
\lim_{k\rightarrow\infty}\frac{1}{k^{r}}\bk(z)=\frac{[X^1:X] \mathrm{rank}(\mathcal{W}_{\uk})}{(4\pi)^{r}};
\end{align*}
and when the Shimura variety $X$ is noncompact, we have the following inequality
\begin{align*}
\limsup_{k\rightarrow\infty}\frac{1}{k^{d}}\bk(z)\leq\frac{[X^1:X]}{(4\pi)^{d}}.
\end{align*}
\end{theorem}
Asymptotic estimates of Bergman kernels associated to holomorphic line bundles on compact complex manifolds has been an object of extensive study over the years, and Tian, Demailly,  Zelditch, et al., have made stellar contributions to this theory.  We have proved the above theorem by specializing the results of Bouche and Berman to our setting, which are an extension of the foundational work of Demailly and Zelditch.
\begin{rem}
The reason why we only work with parallel weights corresponding to the unramified part of the 
weight is the following. Let us consider the situation of nonparallel weights, i.e., let  $\underline{k}_{\spl}=(k,\dots,k, k_{1},\dots,k_{i})$, for some fixed vector $(k_{1},\dots k_{i})\in\mathbb{Z}_{>0}^{i}$ and $i>0$. Let $\bk(z)$ be the Bergman kernel associated to vector the space of automorphic cusp forms $\mathcal{S}_{\uk,\underline{v}}(\Gamma)$. Then, from the proof of Main Theorem \ref{mainthm1}, for any $z\in X$, we can easily show that
\begin{align*}
 \lim_{k\rightarrow\infty}\frac{1}{k^{r}}\bk(z)=0.
\end{align*}\end{rem}

We now state the second main result of the article, which is an extension of the the estimates  from \cite{jk2} to classical Hilbert modular cusp forms, and is proved as Theorems \ref{proofmainthm2}  and \ref{proofmainthm2.2} in section \ref{results2}. 
\begin{theorem}\label{mainthm2}
Let notation be as in section \ref{notation1}, and let the quaternion algebra be split at all finite places of $F$. For any $\uk=(k_1,\ldots,k_r)\in (2\mathbb{Z}_{>0})^r$ and $z\in X^{1}$, when $\Gamma^1$ is of {\it{Type (1)}}, we have the following estimate
\begin{align}\label{thm2estimate1}
 \sup_{z\in X^1}\mathcal{B}_{X^{1}}^{\uk}(z)=O_{X^1}\left(\prod_{j=1}^{r}k_j\right);
\end{align}
and when $\Gamma^1$ is {\it{Type (2)}}, we have following estimate
\begin{align}\label{thm2estimate2}
 \sup_{z\in X^1}\mathcal{B}_{X^1}^{\uk}(z)=O_{X^1}\Bigg(\prod_{j=1}^{d}k_j^{3\slash 2}\Bigg),
\end{align}
where the implied constants in both the above estimates are dependent on $X^1$. 
\end{theorem}
\begin{rem}
However the implied constants in both the above estimates remain stable in covers of Hilbert modular varieties. Furthermore, when $\Gamma^1$ is of {\it{Type (2)}}, the implied constant in estimate \eqref{thm2estimate2} is bounded by a constant which depends only on $\mathrm{PSL}_{2}(\mathcal{O}_{F})$. Hence, the implied constant in  estimate \eqref{thm2estimate2} is a universal constant. 
\end{rem}

We now state the third main result of this article, which is proved as Corollary \ref{cor1} in section \ref{results1}. 
\begin{theorem}
Let notation be as in section \ref{notation1}. With notation as above, let the Shimura variety $X$ be compact. Then, for $k\in k_{\circ}\mathbb{Z}_{>0}$, $\uk = (k, \dots, k, \uk_\ram)$ and $z\in X$, 
we have 
\begin{align}\label{que}
\lim_{k\rightarrow \infty}\frac{1}{\delta_{\uk}}\bk(z)\hypnvol(z)=\shyp(z), 
\end{align}
and the convergence of the limit is uniform in $z\in X$.
\end{theorem}
This proves an average version of the holomorphic QUE conjecture for automorphic cusp forms associated to quaternion algebras. 
\begin{rem}
Adapting results from \cite{dai} (Theorem 1.1), it is easy to show that
\begin{align}\label{rateofconv}
\frac{1}{\delta_{\uk}}\bk(z)= \shyp(z)+O_{X}(k^{-r}),
\end{align}
where the implied constant depends on the Shimura variety $X$. A careful analysis of the adaption of Theorem 1.1 in \cite{dai} to our setting should enable one to prove that the implied constant in equation \eqref{rateofconv} is independent of the Shimura variety $X$.  
\end{rem}
In \cite{liu}, Liu has proved the above equidistribution result, for Hilbert modular cusp forms of even weight associated to the full modular group $\mathrm{SL}_{2}(\mathcal{O}_{F})$. Let $X:=\mathrm{SL}_{2}(\mathcal{O}_{F})\backslash\calh^d$ be the associated Hilbert modular variety. Then, with notation as above, for any compact subset $A\subset X$, Liu has proved that
\begin{align}
\int_{A}\bk(z)\mu_{\mathcal{H}^{d}}^{\mathrm{vol}}(z)=\int_{A}\shyp(z)+O_{\varepsilon,A}(k^{(-1+\varepsilon)d}),
\end{align}
for any $0<\varepsilon<1$. Liu used an infinite series representation for the Bergman kernel associated to the space of cusp forms. Using Liu's technique, in \cite{codgell}, Codgell and Luo have extended the above equidistribution result to Siegel modular cusp forms associated to $\mathrm{Sp}_{2n}(\mathbb{Z})$. 

In \cite{anil1}, the first named author has proved equation \eqref{que} for cusp forms associated to a cocompact Fuchsian subgroup $\mathcal{G}$ of $\mathrm{PSL}_{2}(\mathbb{R})$.  

In \cite{linden}, in a seminal work, Lindenstrauss has proved the AQUE conjecture for M\"ass forms associated to division quaternion algebras. In \cite{sound}, Holowinsky and Soundararajan have proved the quantum unique ergodicity conjecture for Hecke eigenforms of even weight with respect to $\mathrm{SL}_{2}(\mathbb{Z})$. In \cite{marshall}, Marshall has extended methods of Holowinsky and Soundararajan to prove QUE for Hecke eigenforms on $\mathrm{GL}_2$ over a totally real number field, and to automorphic forms of cohomological type on $\mathrm{GL}_2$ over an arbitrary number field, which satisfy certain Ramanujan bounds. The conjecture is still open, when the associated analytic space is compact.  
\begin{rem}
Our methods are very analytic and also some what geometric, and easily extend to automorphic cusp forms associated to irreducible lattices. However, we restrict ourselves to arithmetic groups for the applications of our results to number theory.  
\end{rem}

\vspace{0.2cm}\noindent
{\bf{Organization of the article}}. 
In the second section, we introduce technical results from literature, which we later use in the fourth section. In section \ref{subsecbergman}, we introduce Bergman kernels associated to vector bundles defined over complex manifolds, and state the main results regarding the estimates of Bergman kernels from \cite{bouche} and \cite{berman}, which are used in section \ref{results1}. In section \ref{subsechkestimates}, we introduce weight-$k$ heat kernels associated to hyperbolic Riemann surfaces, and state the estimates of weight-$k$ heat kernels from \cite{jk2}. We use these estimates in section \ref{results2}

In the third section, we recall the basics of automorphic forms defined over quaternion algebras in detail.  In the fourth section, using results from section two, we derive both qualitative and quantitative estimates of the Bergman kernel associated to automorphic forms.  In section \ref{results1}, we derive qualitative estimates of the Bergman kernel associated to automorphic cusp  forms defined over quaternion algebras, and prove an average version of the holomorphic QUE conjecture. In section \ref{results2}, we derive quantitative estimates of the Bergman kernel associated to classical Hilbert modular cusp forms. 
\bigskip
\section{Bergman kernel on complex manifolds and Galois covers and heat kernel estimates}\label{sectionbergman}
In this section, we recall the main results from \cite{bouche}, \cite{berman}, \cite{ma}, and \cite{jk1} 
which we use in section \ref{results1}. 
\subsection{Bergman kernels on complex manifolds}\label{subsecbergman}
Let $M$, $\widetilde{M}$ be  compact manifolds, with $\pi:\widetilde{M}\longrightarrow M$ being a finite Galois cover of $M$. Let 
$E$ denote the group of deck transformations for the the cover $\pi:\widetilde{M}\longrightarrow M$. Let $\mathcal{W}$ be a Hermitian vector bundle with a Hermitian metric $\|\cdot\|_{\mathcal{W}}$ on $M$. Let  $\lbrace s_{i}\rbrace$ denote an orthonormal basis for the space of holomorphic global sections $H^{0}(M,\mathcal{W})$. For $z\in M$, the following function is called the Bergman kernel associated to the vector bundle $\mathcal{W}$
\begin{align}\label{bkdefn}
\mathcal{B}_M^{\mathcal{W}}(z):= \sum_{i}\| s_{i}\|_{\mathcal{W}}^{2} (z).
\end{align}
The above definition is independent of the choice of orthonormal basis. There is a similar definition of Bergman kernel for any holomorphic vector bundle on $\widetilde M$, in particular, for $\pi^* \mathcal{W}$. It is a natural question to ask if there is a relationship between the Bergman kernels $\mathcal{B}_M^{\mathcal{W}}$ and $\mathcal{B}_{\widetilde{M}}^{\pi^* \mathcal{W}}$.

We can say something when $\mathcal{W}$ is of the following special form. Let $\mathcal{G}$ and $\mathcal{J}$ be a Hermitian vector bundle and a Hermitian line bundle, respectively over $M$. For $k\in\mathbb{Z}_{>0}$, let $\mathcal W_k := \mathcal G \otimes \mathcal L^{\otimes k}$.  We have the following asymptotic  relation between the Bergman kernels (see \cite[Th.\,2]{ma}),
\begin{align}\label{maeqn1}
\mathcal{B}_{M}^{\mathcal W_k}(\pi(z)) =\sum_{\gamma\in E}\mathcal{B}_{\widetilde{M}}^{\pi^* \mathcal{W}_k}(\gamma z),
\end{align}
for $k \gg 0$. However, we require a slightly modified version of the above identity. Let $\lbrace \mathcal{W}_{k}\rbrace_{k\in\mathbb{Z}_{>0}}$ be a family of vector bundles defined over $M$, such that for each $k\in\mathbb{Z}_{>0}$, we have  $\pi^{\ast}\mathcal{W}_{k}= \mathcal{F} \otimes \mathcal{L}^{\otimes k}$ on $\widetilde{M}$. In other words, it is sufficient to assume that the vector bundle decomposes into a tensor product of above form only after pullback to $\widetilde M$.  Let $z\in M$, and let $\tilde{z}\in\pi^{-1}(z)$. Then, from the proof of Theorem 2 in \cite{ma} on p.1336, where equation \eqref{maeqn1} is proved, it is clear that for $k\in\mathbb{Z}_{>0}$ with $k\gg 0$, we have 
\begin{align}\label{maeqn}
\mathcal{B}_{M}^{\mathcal{W}_{k}} (z)=\sum_{\gamma \in E} \mathcal{B}_{\tilde{M}}^{\mathcal{F} \otimes \mathcal{L}^{\otimes k}}(\gamma \tilde{z}).
\end{align}
We now recall some estimates on the growth of the Bergman kernel on $\widetilde M$. Let the Hermitian metric on $\mathcal L$ be of the form
$\|s(z)\|^{2}_{\mathcal{L}}:=e^{-\phi(z)}|s(z)|^{2}$, where $s\in\mathcal{L}$ is any section, and $\phi(z)$ is 
a real-valued function defined on $\widetilde M$. Let 
\begin{align}\label{curvatureform}
c_{1}(\mathcal{L})(z):=\frac{i}{2\pi}\partial\overline{\partial}\phi(z)
\end{align}
denote the curvature form of the line bundle $\mathcal{L}$ at the point $z\in \widetilde M$. Let $\alpha_{1},\ldots,\alpha_{n}$ 
denote the eigenvalues of $\partial\overline{\partial}\phi(z)$ at the point $z\in \widetilde M$. Then, in  Theorem 2.1 in \cite{bouche}, for $z\in \widetilde M$, Bouche derived the following equality 
\begin{align}\label{boucheeqn2}
\lim_{k\rightarrow\infty}\frac{1}{k^{n}} \mathcal{B}_{\widetilde M}^{\mathcal F \otimes \mathcal{L}^{\otimes k}} (z)= \mathrm{rank}(\mathcal{F})\cdot\mathrm{det}_{\omega}\big(c_{1}(\mathcal{L})(z)\big),
\end{align}
and the convergence of the above limit is uniform for $z\in \widetilde M$. 

When $\widetilde M$ is a noncompact complex manifold, using micro-local analysis of the Bergman kernel, in \cite{berman}, for $z\in \widetilde M$, Berman derived the following inequality
\begin{align}\label{bermaneqn}
\limsup_{k\rightarrow\infty}\frac{1}{k^{n}}\mathcal{B}_{\widetilde M}^{\mathcal{L}^{\otimes k}}(z)\leq \mathrm{det}_{\omega}\big(c_{1}(\mathcal{L})(z)\big).
\end{align}
\subsection{Estimates of heat kernels on Hilbert modular varieties}\label{subsechkestimates}
Let the quaternion algebra $B$ be split at all finite places of $F$, which implies that $\Gamma^1$ is a discrete and irreducible arithmetic subgroup of $\mathrm{PSL}_{2}(\mathbb{R})^{r}$, as in section \ref{notation1}. Recall that the quotient space $X^1=\Gamma^1\backslash\calh^{r}$ is a complex manifold of finite hyperbolic volume. For any $\uk:=(k_1,\ldots,k_r)\in (2\mathbb{Z}_{>0})^r$ and $z=(z_1,\ldots,z_r)\in\calh^r$, let
\begin{align*}
\Delta_{\mathrm{hyp}}^{\uk}:=\sum_{j=1}^{r}\Delta_{\mathrm{hyp},j}^{k_j},\,\,\mathrm{where}\,\,\Delta_{\mathrm{hyp},j}^{k_j}:=-y_j^{2}\Bigg(\frac{\partial^{2}}{\partial x_{j}^{2}}+\frac{\partial^{2}}{\partial y_{j}^{2}}\Bigg)+2ik_jy_j\frac{\partial }{\partial x_j},
\end{align*}
denote the hyperbolic Laplacian of weight-$\uk$ on $\calh^r$.

For any $k_j> 0$ and $z_j, w_j\in\calh$, let $K_{\calh}^{k_j}(t;z_{j},w_j)$ denote the hyperbolic heat kernel associated to $\Delta_{\mathrm{hyp}}^{k_j}$, the hyperbolic Laplacian of weight-$k_j$ on $\calh$. Then, for any  $\uk=(k_1,\ldots,k_r)\in (2\mathbb{Z}_{>0})^r$, $t\in\mathbb{R}_{>0}$, and $z=(z_1,\ldots,z_r), w=(w_1,\ldots,w_r)\in\calh^{r}$, put
\begin{align*}
K_{\calh^r}^{\uk}(t;z,w):=\prod_{j=1}^{r}K_{\calh}^{k_j}(t;z_{j},w_j),
\end{align*}
which is the hyperbolic heat kernel of weight-$\uk$ on $\calh^r$. 

Now, for any  $\uk=(k_1,\ldots,k_r)\in (2\mathbb{Z}_{>0})^r$, $t\in\mathbb{R}_{>0}$, and $z=(z_1,\ldots,z_r), w=(w_1,\ldots,w_r)\in X^1$ the hyperbolic heat kernel of weight-$\uk$ associated to $X^1$ is given by the following formula
\begin{align*}
K_{X^1}^{\uk}(t;z,w):=\sum_{\gamma=(\gamma_1,\ldots,\gamma_r)\in\Gamma^1}\,\prod_{j=1}^{r}\Bigg(\frac{c_j\overline{w}_j+d_j}{c_jw_j+d_j}\Bigg)^{k_j}\Bigg(\frac{z_j-\gamma_j \overline{w}_j}{\gamma_j w_j-\overline{z}_j}\Bigg)^{k_j}K_{\calh}^{k_j}(t;z_{j},w_j).
\end{align*}
 For any $\uk=(k_1,\ldots,k_r)\in (2\mathbb{Z}_{>0})^r$, let $\mathcal{S}_{\uk}(\Gamma^1)$ denote the complex vector space of weight-$\uk$ cusp forms. Recall that from section \ref{notation1}, for  any $z=(z_1,\ldots,z_r) \in X^{1}$, the Bergman kernel associated to the vector space $\mathcal{S}_{\uk}(\Gamma^1)$ is denoted by $\mathcal{B}_{X^1}^{\uk}(z)$, and is given by equation \eqref{defnbergmanclassical}

For any $\uk=(k_1,\ldots,k_r)\in (2\mathbb{Z}_{>0})^r$, following the same arguments as in \cite{jk2}, we arrive at 
\begin{align*}
\mathcal{S}_{\uk}(\Gamma^1)\subseteq \mathrm{Ker}\bigg(\Delta_{\mathrm{hyp}}^{\uk}-\sum_{j=1}^{r}k_j(1-k_j)\bigg).
\end{align*}
 
Furthermore, for any $t\in\mathbb{R}_{>0}$,  we have the following inequality
\begin{align}\label{hkeqn0}
\mathcal{B}_{X^1}^{\uk}(z)\leq \lim_{t\rightarrow \infty} e^{-\sum_{j=1}^{r}k_j(k_j-1)t}\cdot K_{X^1}^{\uk}(t;z,z)\leq  K_{X^1}^{\uk}(t;z,z)\notag\leq \\
\sum_{\gamma=(\gamma_1,\ldots,\gamma_r)\in\Gamma^1}\,\prod_{j=1}^{r}K_{\calh}^{k_j}(t;z_j,\gamma_jz_j) .
\end{align}
 Now following the same arguments as in \cite{jk2}, namely, inequalities (11) and (13), for any  $\uk=(k_1,\ldots,k_r)\in (2\mathbb{Z}_{>0})^r$, $\gamma=(\gamma_1,\ldots,\gamma_r)\in\Gamma^1$, and $z=(z_1,\ldots,z_r)\in\calh^{r}$, we have the following inequality
 \begin{align}\label{hkeqn1}
K_{\calh}^{k_j}(t;z_{j},\gamma z_j)\leq  \frac{k_j^2}{\sqrt{2}\pi(k_j+1\slash 2)} \cdot\frac{1}{\cosh^{2k_j}({\rho_{\gamma_j, z_j}}/{2})} \int_{\rho_{\gamma_j, z_j}}
 ^{\infty} \frac{r e^{-r/2}dr}{\sqrt{\cosh(r)- \cosh({\rho_{\gamma_j, z_j}})}},
\end{align}
where $\rho_{\gamma_j, z_j} := \mathrm{d}_{\calh}(z_j, \gamma_j z_j)$. Now, for any $u\geq 0$, using the fact that
\begin{align*}
\cosh^{2k_j}(u) \geq \cosh^2 (u) \geq\frac{ e^{2u}}{4 },
\end{align*}
and combining it with inequality \eqref{hkeqn1}, we have
\begin{align}\label{hkeqn2}
 K_{\calh}^{k_j}(t;z_{j},\gamma_j z_j) \leq \frac{2\sqrt{2}k_j^2\cdot e^{-\rho_{\gamma_j, z_j}}}{\pi(k_j+1\slash 2)} \cdot \int_{\rho_{\gamma_j, z_j}}^{\infty}
  \frac{r e^{-r/2}dr}{\sqrt{\cosh(r)- \cosh({\rho_{\gamma_j, z_j}})}} .
\end{align}
From inequality (2.5) in \cite{abms}, we have the following inequality  
\begin{align}\label{hkeqn3}
\int_{\rho_{\gamma_j, z_j}}^{\infty}\frac{r e^{-r/2}dr}{\sqrt{\cosh(r)- \cosh({\rho_{\gamma_j, z_j}})}} \leq 2\sqrt{2}e^{-\rho_{\gamma_j, z_j}}.
\end{align}
 
Combining inequalities \eqref{hkeqn2} and \eqref{hkeqn3}, we have the following inequality  
\begin{align}\label{hkeqn4}
 K_{\calh}^{k_j}(t;z_{j},\gamma z_j) \leq \frac{8k_j^2\cdot e^{-2\rho_{\gamma_j, z_j}}}{\pi(k_j+1\slash 2)} .
\end{align}
Hence, combining inequalities \eqref{hkeqn0} and \eqref{hkeqn4}, we arrive at the following inequality
 \begin{align}\label{hkeqn}
 \mathcal{B}_{X^1}^{\uk}(z)\leq 
\sum_{\gamma=(\gamma_1,\ldots,\gamma_r)\in\Gamma^1}\,\prod_{j=1}^{r}K_{\calh}^{k_j}(t;z_j,\gamma_jz_j) \leq \sum_{\gamma=(\gamma_1,\ldots,\gamma_r)\in\Gamma^1}\,\prod_{j=1}^{r}\frac{8k_j^2\cdot e^{-2\rho_{\gamma_j, z_j}}}{\pi(k_j+1\slash 2)}.
\end{align}

Finally, we now define injectivity radius of a Hilbert modular variety, and state an inequality which will be very useful in section \ref{results2}. When $\Gamma^1$ is of {\it{Type (1)}}, put
\begin{align}\label{injecttype1}
r_{X^1} := \inf \{ \rho_{\gamma_j, z_j} \big|\, 1 \leq j \leq r,\, z=(z_1,\ldots,z_r) \in \calh^r, \,\gamma=(\gamma_1,\ldots,\gamma_r) \in \Gamma^1 \setminus \mathrm{Id}\}.
\end{align}

When $\Gamma^1$ is of {\it{Type (2)}}, let $\mathcal{C}$ denote the set of cusps of $\Gamma^1$, and for any $p\in\mathcal{C}$, let $\Gamma^{1}_{p}$ denote the stabilizer of the cusp $p$ in $\Gamma^1$ . Now put
\begin{align}\label{injecttype2}
r_{X^1} := \inf \bigg\lbrace \rho_{\gamma_j, z_j} |\,1 \leq j \leq d,\,z=(z_1,\ldots,z_d) \in \calh^d,\, \gamma=(\gamma_1,\ldots,\gamma_d) \in \Gamma^1\setminus\bigg(\sum_{p\in\mathcal{C}} \Gamma^1_{p}\bigg)\bigg\rbrace .
\end{align}
 
We now state  an adaptation of an inequality from \cite{jl} to our setting. For any positive, smooth, real-valued, and decreasing function $f$ defined on $\mathbb{R}_{\geq 0}$, and for any $\delta > r_{X^1}\slash 2$, we have the following inequality 
\begin{align}\label{jlineq}
 \int_0^{\infty} f(\rho) dN_{\Gamma}(z_j;\rho) \leq \int_0^{\delta} f(\rho) dN_{\Gamma}(z_j;\rho) + f(\delta) \frac{\sinh(r_{X^1}/2)\sinh(\delta)}{\sinh^2(r_{X^1}/4))}
 +\notag\\\frac{1}{2\sinh^2(r_{X^1}/4)} \int_{\delta}^{\infty} f(\rho) \sinh(\rho + r_{X^1}/2) d\rho,
\end{align}
where 
\begin{align*}
dN_{\Gamma}(z_j;\rho) := \mathrm{card}\,\bigg\{ \gamma |\,\gamma \in \Gamma^1\setminus\bigg(\sum_{p\in\mathcal{C}} \Gamma^1_{p}\bigg),\, \rho_{\gamma_j , z_j} \leq \rho\bigg\}.
\end{align*} 
\bigskip
\section{Cusp forms on quaternion algebras}\label{secauto}
\noindent
Let $F$ be a totally real field of degree $d$ over $\Q$. Let $\calo_F$ denote the ring of integers of $F$. Let $\Sigma=\Sigma_F$ denote the set of infinite places of $F$.  

We will first briefly recall some basic facts about quaternion algebras. See \cite{VignerasBook} for more  details. A quaternion algebra $B$ over $F$ is a central simple algebra of degree $4$ over $F$. Let $v$ be a place of $F$, and $B$ is said to be split or unramified at $v$ if $B_v := B \otimes_F F_v$ is isomorphic to the matrix algebra $M_2 (F_v)$. If $B_v$ is not isomorphic to the matrix algebra, $B$ is ramified at $v$. Take elements $a,b \in F^\times$ and let $B$ be the algebra generated by $i,j$ and $k$ satisfying $i^2 =a, j^2=b $ and $k = ij = -ji$. This is a quaternion algebra and in fact any quaternion algebra can be constructed this way. A place $v$ is unramified for $B$ if and only if the local Hilbert symbol $(a,b)_v=1$. Local global compatibility implies that the number of places where $B$ is ramified is a finite set with even number of elements.

Let $S$ be a finite set of places of $F$ such that $|S|$ is even. Up to isomorphism, this set determines uniquely a quaternion algebra $B$ that is ramified at precisely the primes in $S$ and unramified outside. Let $\Sigma_\ram := \Sigma \cap S$ be the infinite primes where $B$ is ramified  and $\Sigma_\spl := \Sigma \setminus \Sigma_\ram$ be the infinite places where $B$ is unramified. For any finite place $v$ of $F$, $B_v = B \otimes_F F_v$ is a central simple algebra of degree $4$  over $F_v$ and there are only two possibilities, namely, the algebra $M_2 (F_v)$ or the algebra corresponding to $1/2$ in the Brauer group $\mathrm{Br} (F_v) \cong \Q/\Z$.  For any $v \in \Sigma$, there are again only two possibilities for $B_v = B \otimes_F \R$, namely, $M_2 (\R)$ or the usual Hamiltonian quaternions 
\begin{align} \label{haqua}
\H= \{ a+bi+cj+dk \mid a,b,c,d \in \R, \text{with } i^2=j^2=k^2=ijk=-1\}.
\end{align}
By definition, $B_v = M_2 (\R)$ for $v \in \Sigma_\spl$ and $B_v = \H$ for $v \in \Sigma_\ram$. Let $r := |\Sigma_\spl|$ be the number of unramified infinite places. When $r=0$, the quaternion algebra is called definite. We will assume throughout this article that this is not the case, i.e., $r > 0$. The reason for this assumption is that the locally symmetric space we consider below will be a finite set of points when $B$ is definite and is not amenable to the analytic techniques we use later. 

Let $G$ be the restriction of scalars from $F$ to $\Q$ of the algebraic group associated to the units in $B$. Concretely, let $A$ be any $\Q$-algebra, then the $A$ points of $G$ are given by 
$$ G(A) = ( B \otimes_F (F \otimes_\Q A))^\times. $$
So $G$ is an algebraic group defined over $\Q$. Note that since it is possible to find a quadratic extension $K/F$ such that $B \otimes_F K \cong M_2 (K)$, the group $G$ is in fact a twisted form of $\GL_2$. Let $\A$ be the adele ring over $\Q$ and let $\A_f$ denote the finite part of the adeles. Let $G(\A)$ be the adelic points of $G$ which is defined as $\prod^\prime G(F_v)$, where the restricted product is over $G(\calo_{F,v})$ at all the finite places $v \not \in S$. Here $\calo_{F,v}$ is the valuation ring of $F_v$. Let $G_\infty = G(\R) = (B \otimes_F \R)^\times = \prod_{v \in \Sigma} B_v^\times$ the infinite part of $G$. It follows that $G_\infty = (\GL_2 (\R))^r \times (\H^\times)^{d-r}$. Let $K_\infty$ be the maximal compact subgroup of $G_\infty$. It is equal to $(O_2 (\R))^{r} \times (\H^1)^{d-r} $, where $\H^1$ is the subgroup of norm one elements in $\H$. Let $Z_\infty$ be the center of $G_\infty$. For an open compact subgroup $U \subset G(\A_f)$, the associated locally symmetric space that we are interested in is given by
$$ 
\Sh_U = \Sh^G_U = G(\Q) \backslash G(\A) /U (Z_\infty K_\infty)^\circ
$$
where $(Z_\infty K_\infty)^\circ$ is  the connected component of $Z_\infty K_\infty$ containing the identity element. This is the quaternionic Shimura variety attached to $B$ of level $U$. When $B=M_2 (F)$, these are usually referred to as Hilbert-Blumenthal varieties. The Hilbert-Blumenthal varieties are not compact and can be compactified by adding finitely many cusps. Whenever $B \not = M_2(F)$, the associated Shimura variety is compact. 

Let $G^1 \subset G$ be the algebraic group corresponding to the subgroup of norm $1$ elements $B^1 \subset B^\times$. Similar to the construction above, we consider the Shimura variety associated to $G^1$ of level $U^1 = U \cap G^1(\A_f)$ and this is denoted by $\Sh^1$. The natural inclusion $G^1 \subset G$ induces an open, quasi-finite morphism $\Sh^1_{U^1} \to \Sh_U$ at the level of Shimura varieties. 

Let $G_\infty^+ = (\GL_2 (\R)^+)^r \times (\H^\times)^{d-r}$ denote the subgroup of totally positive elements in $G_\infty$ and $K_\infty^+ = G_\infty^+ \cap K_\infty$. Let $\calh$ be the upper half-plane and choose a base point $z_0 = (z^0_{1}, \dots, z^0_{r})$ in $\calh^r$.  It is easy to check that $G_\infty^+/Z_\infty K_\infty^+ \cong \calh^r$  and the isomorphism is given by sending $ g= (g_\spl, g_\ram) \mapsto g_\spl (z_0)$. Here and elsewhere, $g_\spl \in (\GL_2 (\R)^+)^r$ is the component at the split places and $g_\ram \in  (\H^\times)^{d-r}$ is the component at the ramified places.  By strong approximation theorem, we can find finitely many $t_i \in G(\A_f)$ such that $G(\A) = \bigsqcup_{i=1}^h G(\Q) t_i U G_\infty^+$. The integer $h$ is the strict class number of $F$ and is independent of $B$ and $U$ under our assumption that $r > 0$.  This gives a decomposition of the Shimura variety into connected components,
$$
\Sh_U = \bigsqcup_{i=1}^h \Gamma_i \backslash \calh^r 
$$
by sending $\gamma t_i u g_\infty \mapsto (g_\infty)_\spl (z_0)$ for $\gamma \in G(\Q), uf \in U$ and $g_\infty \in G_\infty^+$. Here the groups $\Gamma_i = \Gamma_i (U) = G(\Q)^+ \cap t_i U G_\infty^+ t_i^{-1}$, which are also viewed as discrete subgroups of $(\GL_2(\R)^+)^r$ acting on $\calh^r$. Similarly, for the Shimura variety $\Sh^1_{U^1}$, we have a decomposition,
$$
\Sh^1_{U^1} = \bigsqcup_{i=1}^h \Gamma^1_i \backslash \calh^r,
$$
where $\Gamma^1_i = \Gamma_i \cap G^1 (\Q)$. In fact, the projection map $\pi_i : \Gamma^1_i \backslash \calh^r \to \Gamma_i \backslash \calh^r $ on the $i$th connected component is a finite \'etale covering with Galois group $\Gamma_i / \Gamma^1_i (F^\times \cap \Gamma_i)$. Note that all these Galois groups are conjugate to each other.

A weight for $G$ is an algebraic character on the maximal torus of $G$ at infinity and we identify this with an element $\uk = (k_\sigma) \in \Z [\Sigma] = \Z^d$. Let $\underline{1} = (1, \dots, 1) \in \Z[\Sigma]$ and a weight is called parallel if it is an integer multiple of $\underline{1}$. There is a partial ordering of the weights by formula $\uk \ge \uk^\prime$ if $k_\sigma \ge k^\prime_\sigma$ for each $\sigma \in \Sigma$.  We will assume that $\uk \ge \underline{2}$. Fix a parallel defect of $\uk$, i.e., a $0 \le \uv \in \Z[\Sigma]$ such that $\uk + 2\uv$. Note that the existence of the parallel defect $\uv$ is equivalent to saying that all the $k_\sigma$ have the same parity. This is what we call the parity condition on the weights. Without this assumption, the space of weight-$\uk$ cusp forms is zero.

For a weight $\uk$ as above, we again break it up as $(\uk_\spl, \uk_\ram)$, with $\uk_\spl \in \Z[\Sigma_\spl]$ and $\uk_\ram \in \Z[\Sigma_\ram]$.  For $\underline{z} = (z_1, \dots, z_d) \in \C^\times$ and weight $\uk$, we define $\underline{z}^{\uk} = \prod_{i=1}^d z_i^{k_i}$.

We now define certain representations that are needed to define automorphic forms. The underlying vector space of the representation $L(\uk_\ram, \uv_\ram, \C)$  is the complex vector space of polynomials in $2(d-r)$ variables $\{ X_\sigma, Y_\sigma \}_{\sigma \in \Sigma_\ram}$, which are homogeneous of degree $k_{\sigma} - 2$ in the variables   $X_\sigma, Y_\sigma$, for each $\sigma \in \Sigma_\ram$. This has dimension $\prod_{\sigma \in \Sigma_\ram} (k_\sigma - 1)$ over $\C$.  Given $P \in L(\uk_\ram, \uv_\ram, \C)$, there is a right action of $(\H^\times)^{d-r}$ on $P$ by viewing it as a subgroup of $\GL_2 (\C)^{d-r}$. Using this, define an action of $x =(x_f, x_\infty) \in G(\A)$ by
$$ 
P|_x = \nu (x_{\infty, \ram})^{v_\ram} P|_{x_{\infty, \ram}},
$$
where $\nu$ is the reduced norm map on the quaternions. This is the representation denoted by $L(k_\ram, v_\ram, \C)$. We can also convert this to a left action by setting $x \cdot P = P|_{x^{-1}}$. 

Our main references for definitions and basic properties of automorphic forms on $G$ will be \cite[\S 1]{Shimura81}, \cite[\S 2]{Hida88} and \cite{TilouineNotes}. 

An automorphy factor is a function $j: G_\infty^+ \times \calh^r \to (\C^\times)^{r}$ satisfying the cocycle condition
$$ 
j (g_1 g_2, z) =  j(g_1, g_2 (z)) j (g_2, z), \quad \forall g_1,g_2 \in G_\infty^+.
$$
We are interested in the automorphy factor associated to the weight-$\uk$ given by the formula $ j_{\uk} (g,z) = \prod_{\sigma \in \Sigma_\spl} (c_\sigma z_\sigma + d_\sigma)^{k_\sigma}$. 
A cusp form is a function $f: G(\A) \to L (\uk_\ram, \uv_\ram, \C)$ such that
\begin{itemize}
\item $f(\gamma g) = f(g)$, for all $\gamma \in G(\Q)$,
\item  $f|_{u} = f $, for all $u \in U G_\infty^+$, where the slash operator is defined as
$$ f|_{u} (g) = j_{\uk} (u_\infty, z_0)^{-1} \nu (u_\infty)^{(\uv + \uk - \underline{1})_\spl} [f(gu^{-1}) |_{u_\infty}]. $$
Here $\nu$ is again the reduced norm on $G_\infty$.
\item appropriate holomorphicity condition, see \cite[2.4b]{Hida88}, and
\item when $B = M_2 (F)$ a cuspidality condition and if $F=\Q$, there is in addition a uniform boundedness condition.
\end{itemize}
Let $S_{\uk,\uv} (U, \C)$ denote the complex vector space of all such forms. It is finite dimensional. Similar to the decomposition of the Shimura variety into connected components, we can decompose the space of adelic cusp forms into a direct sum of spaces of classical cusp forms on $\calh^r$. We describe this process now. For any point $z \in \calh^r$, choose an element $g_\infty \in G_\infty^+$ such that $g_\infty (z_0) = z$. Then for any $u \in G(\A_f)$, define a function $f: \calh^r \to L(\uk_\ram, \uv_\ram, \C)$ as
 $$ 
 f_u (z) = j_{\uk} (g_\infty, z_0) \nu (g_\infty)^{-(\uv + \uk - \underline{1})_\spl } [f(ug_\infty)|_{u_\infty^{-1}}].
 $$
 It can be verified that this function does not depend on the choice of $g_\infty$ made in the definition. If we let $\phi_i = f_{t_i}$, it is easy to check that 
 \begin{itemize}
 \item $\phi_i (\gamma (z)) = \nu (\gamma)^{-(\uv + \uk - \underline{1})_\spl} j_{\uk} (\gamma, z) [\gamma \cdot f(z)]$ for all $\gamma \in \Gamma_i (U)$. 
 \item $\phi_i$ is holomorphic as a function of $z$, and
 \item appropriate cuspidality condition when $B=M_2(F)$ and a boundedness at the cusps condition for $F=\Q$.
 \end{itemize}

We let $S_{\uk,\uv} (\Gamma_i, \C)$ denote the vector space of all functions that satisfy these conditions. Then the map sending $f \mapsto (\phi_1, \dots, \phi_h)$ defines an isomorphism 
$$ 
S_{\uk,\uv} (U, \C) \cong \oplus_{i=1}^h S_{\uk,\uv} (\Gamma_i (U), \C)
$$
between the two types of cusp forms.

Take $\Gamma$ to be one of the $\Gamma_i$. Let $\phi, \phi^\prime \in S_{\uk,\uv} (\Gamma, \C)$ be two cusp forms. Then, the Petersson inner product $\langle \phi, \phi^\prime \rangle$ is given 
by the following formula
$$
\langle \phi, \phi^\prime \rangle =  \int_{\Gamma \backslash \calh^r} \phi(z) \cdot {^t} \overline{\phi^\prime (z) } \mathrm{Im} (z)^{\uk}\ \hypnvol(z).
$$
Finally extend this definition to adelic cusp forms as follows. If $f \leftrightarrow (\phi_1, \dots, \phi_h)$ and $f^\prime \leftrightarrow (\phi^\prime_1, \dots, \phi^\prime_h)$, then
$$
\langle f, f^\prime \rangle = h^{-1} \sum_{i=1}^h \langle \phi_i, \phi_i^\prime \rangle.
$$

Finally, we describe the connection between cusp forms and holomorphic sections of a Shimura variety on a vector bundle. Let $W_{\uk,\uv}$ be the vector space given by $\C (\uk_\spl) \otimes \nu^{-(\uv + \uk -  \underline{1})_\spl} \otimes L(\uk_\ram, \uv_\ram, \C)$, where 
\begin{itemize}
\item[-] $\C (\uk_\spl)$ is the one dimensional representation of the torus $(\C^\times)^r$ with the action given by $ (t_\sigma) \mapsto \prod_{\sigma \in \Sigma_\spl} t_\sigma^{k_\sigma}$, and
\item[-] the representation $\nu^{-(\uv + \uk - \underline{1})_\spl} \otimes L(\uk_\ram, \uv_\ram, \C)$ of $G_\infty$, where $\GL_2 (\R)^r$ acts by a power of the determinant $\nu$ and $(\H^\times)^{d-r}$ acts on $L(\uk_\ram, \uv_\ram, \C)$ via the left action defined earlier.
\end{itemize}
For $\Gamma$ as above, by putting together the actions above gives a left action on the space $\calh^r \times W_{\uk,\uv}$ by the rule
$$
\gamma \cdot (z, x) = (\gamma z, j_{\uk} (\gamma, z) \cdot \nu^{-(\uv + \uk -\underline{1})_\spl} \cdot \gamma_\ram \cdot x).
$$
Let $\mathcal{W}_{\uk} = \mathcal{W}_{\uk,\uv}$ denote the vector bundle
$$\xymatrix{
\Gamma \backslash \calh^r \times W_{\uk,\uv} \ar^{p}[d] \\
\Gamma \backslash \calh^r
}$$
induced by this action and let $p$ denote the projection map. Note that matrices of the form $\mat \epsilon 0 0 \epsilon \in \Gamma$, where $\epsilon$ is a unit  in $\calo_F$, act trivially on $\calh^r$. So the action of such elements on $W_{\uk, \uv}$ should be trivial.  In fact, the above scalar matrix acts on $W_{\uk, \uv}$  as multiplication by $\epsilon^{\uk_\spl -2(\uv+\uk -\underline{1})_\spl - (\uk+2\uv-\underline{2})_\ram} = \epsilon^{-(\uk + 2\uk -\underline{2})}$; which is a power of $\mathrm{Nm}(\epsilon)$.  So we now assume either that $F^\times \cap \Gamma$ consists only of totally positive units or that $\uk +2 \uv -\underline{2}$ is an even multiple of $\underline{1}$. In either of these cases, the action of these scalar matrices is trivial.  

We can now identify the space of cusp forms with global sections via the isomorphism 
$$
S_{\uk, \uv} (\Gamma, \C) \longrightarrow H^0 (\Gamma \backslash \calh^r, \mathcal{W}_{\uk,\uv})
$$
$$
\phi \mapsto (s: z \mapsto (z, \phi(z))).
$$

\vspace{0.15cm} \noindent
The decomposition of the vector space $W_{\uk, \uv}$ does not imply that the vector bundle $\mathcal{W}_{\uk, \uv}$ decomposes into a tensor product of a vector bundle and a line bundle. This is precisely because of the presence of units in $F^\times \cap \Gamma$ as such elements do not act trivially on each of the components.  We need to pullback the vector bundle to the finite Galois cover $\Gamma^1 \backslash \calh^r$ considered earlier. When the vector bundle $\mathcal W_{\uk, \uv}$ is pulled back to this space, we get the vector bundle $\Gamma^1 \backslash \calh^r \times W_{\uk, \uv}$ where the action is just restriction from $\Gamma$ to $\Gamma^1$. Now when restricted to $\Gamma^1$, the action simplifies because the determinant terms vanish and since $F^\times \cap \Gamma^1 = 1$, there is indeed a tensor product decomposition of $\pi^* \mathcal W_{\uk} = \mathcal F \otimes \mathcal L^{\otimes \uk}$ into a vector bundle $\mathcal F$ and a power of a line bundle $\mathcal L$. Note that the vector bundle $\mathcal F$ depends only on $\uk_\ram$ and the line bundle depends only on $\uk_\spl$ and neither of them is dependent on the choice of $\uv$.

\bigskip
\section{Estimates of cusp forms}
\noindent
In this section, we prove all the three main theorems that we stated in section \ref{statementofresults}.   
\subsection{Asymptotic estimates of cusp forms}\label{results1}
In this section, using the results from section \ref{subsecbergman}, we derive asymptotic estimates of the Bergman kernel associated to automorphic cusp forms  defined over quaternion algebras, and prove an average version of the holomorphic QUE conjecture. 

As in the previous sections, let $X = \Gamma \backslash \calh^r$ denote a connected component of the Shimura variety $\Sh_U$ and let $X^1 = \Gamma^1\backslash \calh^r$ denote the finite Galois cover of degree $[X^1:X]$ with Galois group $E$. Let $\mathcal W_{\uk,\uv}$ denote the vector bundle  whose space of global sections is isomorphic to the space of cusp forms of  weight-$(\uk,\uv)$. As mentioned earlier, we are interested in the pullback of $\mathcal W_{\uk}$ to $X^1$ and the pullback doesn't depend on $\uv$, so we suppress this from the notation of the vector bundle. 

The Hermitian metric $\|\cdot\|_{\mathcal{W}_{\uk}}$ on $H^{0}\big(X,\mathcal{W}_{\uk})$ is the Petersson metric 
$\|\cdot\|_{\mathrm{pet}}$ defined in equation \eqref{petmetric} on $\mathcal{S}_{\uk,\underline{v}}(\Gamma)$. Hence, from the definition of Bergman kernel $\bk(z)$ associated to the vector space $\mathcal{S}_{\uk,\underline{v}}(\Gamma)$, which is described in equation  \eqref{bcuspdefn}, we can conclude that
\begin{align}\label{eqn1}
\bk(z)= \mathcal{B}_{X}^{\mathcal{W}_{\uk}}(z),
\end{align}
where  $\mathcal{B}_{X}^{\mathcal{W}_{\uk}}(z)$ is the Bergman kernel associated to the vector bundle $\mathcal{W}_{\uk}$ described in equation \eqref{bkdefn}.

\vspace{0.2cm} \noindent
The vector bundle $\pi^{\ast}\mathcal{W}_{\underline{k}_{\circ}}$ splits into $\mathcal{F}\otimes \mathcal{L}$ on $X^1$, where $\mathcal{L}$ is a line bundle, and $\mathcal{F}$ is a vector bundle. Furthermore, for any $k = k_\circ N \in k_{\circ}\mathbb{Z}_{>0}$, the vector bundle $\pi^{\ast}\mathcal{W}_{\uk}$ splits into $\mathcal F \otimes \mathcal L^{\otimes N}$ on $X^1$. As before, we let $\mathcal B_{X^1}^{\mathcal F \otimes \mathcal L^{\otimes N}}$ be the associated Bergman kernel for $z \in X^1$.

\noindent
Let $H^{0}\big( X^1, \calf \otimes \mathcal{L}^{\otimes N} \big)$ denote the space of holomorphic global sections of the line bundle $\mathcal{L}^{\otimes N}$.  For any $f \in H^{0}\big( X^1,  \calf \otimes \mathcal{L}^{\otimes N} \big)$, as in equation \eqref{petmetric}, the  metric $\|f\|_{\mathrm{ \calf \otimes \mathcal{L}^{\otimes N}}}^{2}$ at the point $z=(z_1,\dots,z_r)\in X^{1}$ is defined as
\begin{align}\label{petmetricl}
\|f\|_{\mathrm{\calf \otimes \mathcal{L}^{\otimes N}}}^{2}(z):= \bigg(\prod_{i=1}^{r}y^{k}_{i}\bigg) (f(z) \cdot {^t}\overline{f(z)}).
\end{align}

Let the Shimura variety $X$ be compact. Let $k = k_\circ N$, and $z\in X$ with $z^{1}\in \pi^{-1}(z)$. We now  apply equation \eqref{maeqn} to the compact complex manifold $X$ with the Hermitian vector bundle $\mathcal{W}_{\uk}$, and to the finite Galois cover $\pi:X^{1}\longrightarrow X$ with the Hermitian vector bundle $\pi^{\ast}\mathcal{W}_{\uk}= \mathcal{F}\otimes\mathcal{L}^{\otimes N}$. For $N\gg0$, we have
\begin{align}\label{eqn2}
\mathcal{B}_{X}^{\mathcal{W}_{\uk}}(z)=\sum_{\gamma\in E} \mathcal{B}_{X^{1}}^{\mathcal{F}\otimes\mathcal{L}^{\otimes N}}(\gamma z^{1}).
\end{align}


\noindent
Now suppose that $X$ is non-compact. Then, the vector bundle $\calf$ is trivial and $\mathcal{W}_{\uk} = \mathcal{L}^{\otimes N}$. Let $\lbrace f_{1},\dots,f_{\delta_{\uk}}\rbrace$ denote an orthonormal basis of  $H^{0}\big(X,\mathcal{L}^{\otimes N} \big)$ with respect to  the natural inner-product which induces the Hermitian metric  $\|\cdot\|_{\mathcal{L}^{\otimes N}}$. Here $\delta_{\uk}$ denotes the  dimension of $H^{0}\big(X,\mathcal{L}^{\otimes N} \big)$. Then, we can construct an orthonormal basis $T$ for $H^{0}\big(X^{1}, \mathcal{L}^{\otimes N} \big)$, such that the set $\lbrace f_{1}\slash \sqrt{[X^1:X]},\dots,f_{\delta_{k}}\slash \sqrt{[X^1:X]}\rbrace\subset T$. So let 
\begin{align}\label{basis}
T:=\bigg\lbrace f_{1}^{1}:=\frac{f_{1}}{\sqrt{[X^1:X]}},\dots,f_{\delta_{k}}^{1}:= \frac{f_{\delta_{k}}}{\sqrt{[X^1:X]}},\dots,f_{\delta_{k}^{1}}^{1}\bigg\rbrace 
\end{align}
denote an orthonormal basis of  $H^{0}\big(X^{1},  \mathcal{L}^{\otimes N}\big)$,  where $\delta_{k}^{1}$ denotes the dimension of $H^{0}(X^{1},\mathcal{L}^{\otimes N})$. From the choice of the basis $T$ described in \eqref{basis}, we find that 
\begin{align}\label{eqn32}
\bk(z)\leq [X^1:X] \mathcal{B}_{X^{1}}^{\mathcal{L}^{\otimes N}}(z^1) 
\end{align}
for any $z^1 \in \pi^{-1} (z)$. 

\noindent
Recall that the weights we consider are of the form $\uk=(k\dots,k,\uk_{\ram})\in\mathbb{Z}_{\ge 2}^{r}\times \mathbb{Z}_{\ge 2}^{d-r}$ for $k = k_\circ N$, and the weight-$\uk$ continues to satisfy the parity condition.

\vspace{0.1cm}
\begin{thm}\label{thm1}
Let the notation be as above, and let $X$ be a compact Shimura variety. Then for $z\in X$, we have
\begin{align*}
\lim_{k\rightarrow \infty}\frac{1}{k^{r}}\mathcal{B}_{X}^{\uk}(z)=\frac{[X^1:X] \cdot \mathrm{rank} (\calf)}{(4\pi)^{r}},
\end{align*}
and the convergence of the above limit is uniform in $z\in X$. 
\end{thm}
\begin{proof}
For $z\in X$, let $z^{1}\in \pi^{-1}(z)$. For $N \gg 0$, combining equations \eqref{eqn1} and \eqref{eqn2}, we have 
\begin{align}\label{proofeqn1}
\mathcal{B}_{X}^{\uk}(z)=\mathcal{B}_{X}^{\mathcal{W}_{\uk}}(z)=
\sum_{\gamma\in E} \mathcal{B}_{X^{1}}^{\mathcal{F}\otimes\mathcal{L}^{
\otimes N}}(\gamma z^{1}).
\end{align}
Applying estimate \eqref{boucheeqn2} to the complex manifold $X^{1}$ with its natural Hermitian metric $\hypn$ and the vector bundle $\mathcal{F}\otimes\mathcal{L}^{\otimes N}$, we deduce that 
\begin{align}\label{proofeqn2}
\lim_{k\rightarrow \infty}\frac{k_{\circ}^{r}}{k^{r}} \mathcal{B}_{X^{1}}^{\mathcal{F}\otimes\mathcal{L}^{\otimes N}}(z^{1})= \mathrm{rank} (\calf) \cdot\mathrm{det}_{\hypn}\big(c_{1}(\mathcal{L})(z^{1})\big), 
\end{align}
and the convergence of the above limit is uniform in $z^{1}\in X^{1}$. Now from the definition of Petersson inner-product defined in \eqref{petmetricl}, and from the 
definition of the curvature form described in \eqref{curvatureform}, for  $z^{1}=(z_{1}=
x_{1}+iy_{1},\ldots,z_{r}=x_{r}+y_{r})\in X^{1}$, we have
\begin{align}\label{proofeqn3}
c_{1}(\mathcal{L})(z^{1})=-\frac{i}{2\pi}\partial\overline{\partial}\log\bigg(\prod_{j=1}^{r}y_{j}^{k_{\circ}}\bigg)
=-\frac{i}{2\pi}\sum_{j=1}^{r}\partial\overline{\partial}\log\big(y_{j}^{k_{\circ}}\big).
\end{align}
For any $1\leq j\leq r$, we compute 
\begin{align}\label{proofeqn4}
\frac{i}{2\pi}\partial\overline{\partial}\log\big(y_{j}\big) =
\frac{i}{2\pi}\partial\overline{\partial}\log\left(\frac{z_{j}-\overline{z}_{j}}{2i}\right)=
-\frac{i}{2\pi}\partial\left(\frac{d\overline{z}_{j}}{z_{j}-\overline{z}_{j}}\right)\notag\\=
\frac{i}{2\pi}\cdot\frac{dz_{j}\wedge d\overline{z}_{j}}{\big(z_{j}-\overline{z}_{j}\big)^{2}}
=-\frac{i}{8\pi}\cdot\frac{dz_{j}\wedge d\overline{z}_{j}}{ y_{j}^{2}}=-\frac{1}{4\pi}\hyp(z_{j}).
\end{align}
Combining equations \eqref{proofeqn3} and \eqref{proofeqn4}, we arrive at
\begin{align*}
 c_{1}(\mathcal{L})(z^{1})=\frac{k_{\circ}}{(4\pi)}\hypn(z^{1})\implies
\mathrm{det}_{\hypn}\big(c_{1}(\mathcal{L})(z^{1})\big) =\frac{k_{\circ}^{r}}{(4\pi)^{r}}. 
\end{align*}
Hence, for any $\gamma\in E$ and $z^{1}\in X^{1}$, combining the above equation with equation \eqref{proofeqn2}, we have
\begin{align*}
 \lim_{k\rightarrow \infty}\frac{1}{k^{r}} \mathcal{B}_{X^{1}}^{\mathcal{F}\otimes\mathcal{L}^{\otimes N}}(\gamma z^{1})=\frac{\mathrm{rank}(\calf)}{(4\pi)^{r}},
\end{align*}
and the convergence of the above limit is uniform in $z^{1}\in X^{1}$. Hence, combining above  equation with equation \eqref{proofeqn1}, we arrive 
\begin{align*}
&\lim_{k\rightarrow \infty}\frac{1}{k^{r}}\mathcal{B}_{X}^{\uk}(z)=
\lim_{k\rightarrow \infty}\frac{1}{k^{r}}\sum_{\gamma\in E}
\mathcal{B}_{X^{1}}^{\mathcal{F}\otimes\mathcal{L}^{\otimes N}}(\gamma z^{1})=\\&\sum_{\gamma\in E} \lim_{k\rightarrow \infty}\frac{1}{k^{r}} \mathcal{B}_{X^{1}}^{\mathcal{F}\otimes\mathcal{L}^{\otimes N}}(\gamma z^{1})=\frac{[X^1:X] \cdot \mathrm{rank}(\calf)}{(4\pi)^{r}},
\end{align*}
and the convergence of the above limit is uniform in $z\in X$. This completes the proof of the theorem. 
\end{proof}
\begin{cor}\label{cor1}
Keeping the same notation as above, we have
\begin{align*}
 \lim_{k\rightarrow \infty}\frac{1}{\delta_{\uk}}\mathcal{B}^{\uk}_{X}(z)\hypnvol(z)=\shyp(z),
\end{align*}
and the convergence of the above limit is uniform in $z\in X$. 
\end{cor}
\begin{proof}
From Theorem \ref{thm1}, for $k = k_\circ N$ with $N\gg 0$, and $z\in X$, we have
\begin{align}\label{cor1eqn1}
\mathcal{B}_{X}^{\uk}(z)=C\cdot k^{r}+o(k^{r}),
\end{align}
where $C:=\frac{[X^1:X]  \mathrm{rank}(\calf)}{(4\pi)^{r}}$. This implies that for $k = k_\circ N$ with $N\gg 0$, we have
\begin{align}\label{cor1eqn2}
 \delta_{\uk}=\int_{X}\mathcal{B}_{X}^{\uk}(z)\hypnvol(z)=C\cdot\vx(X)\cdot k^{r}+o(k^{r}).
\end{align}
Combining equations \eqref{cor1eqn1} and \eqref{cor1eqn2}, for $z\in X$, we compute
\begin{align*}
\lim_{k\rightarrow\infty}\frac{1}{\delta_{\uk}}\mathcal{B}_{X}^{\uk}(z)\hypnvol(z)=\frac{1}{\vx(X)}\hypnvol(z)=\shyp(z),
\end{align*}
and the convergence of the limit is uniform in $z\in X$. This completes the proof of the corollary. 
\end{proof}
\begin{cor}\label{cor2}
Let the notation be as above,  and let $X$ be a noncompact Shimura variety. Then, for any $k\in k_{\circ}\mathbb{Z}_{>0}$ and $z\in X$, we have
\begin{align*}
\limsup_{k\rightarrow \infty}\frac{1}{k^{d}}\mathcal{B}_{X}^{\uk}(z)\leq \frac{[X^1:X]}{(4\pi)^{d}}. 
\end{align*} 
\begin{proof}
For any $k= k_\circ N \in k_{\circ}\mathbb{Z}_{>0}$ and $z\in X$, from arguments similar to the ones used in the proof of Theorem \ref{thm1}, and from equation \eqref{bermaneqn}, for $z^1 \in X^{1}$, we have
\begin{align}
 \mathcal{B}_{X^{1}}^{\mathcal{L}^{\otimes N}}(z^{1})\leq \frac{1}{(4\pi)^{d}}. 
\end{align}
Combining the above inequality with inequality \eqref{eqn32} completes the proof of the corollary. 
\end{proof}
\end{cor}
\begin{cor}
Let the notation be as above, and let $X$ be a noncompact Shimura variety. Let $A$ be a compact subset of $X$. Then, for $k \in k_{\circ}\mathbb{Z}_{>0}$ and $z\in X$, we have the following estimate
\begin{align*}
\lim_{k\rightarrow\infty}\frac{1}{k^{d}}\mathcal{B}_{X}^{\uk}(z)= O_{A}(1),
\end{align*}
where the implied constant depends on $A$.
\begin{proof}
From Corollary \ref{cor1}, for $k \in k_{\circ}\mathbb{Z}_{>0}$ and $z\in X$, we have
\begin{align*}
\limsup_{k\rightarrow\infty}\frac{1}{k^{d}}\mathcal{B}_{X}^{\uk}(z)\leq
\frac{1}{(4\pi)^{d}}. 
\end{align*}
As $A\subset X$ is compact, we can find a constant $C$ (which depends on $A$ and is independent of $k$) such that 
\begin{align*}
\lim_{k\rightarrow\infty}\frac{1}{k^{d}}\mathcal{B}_{X}^{k}(z)\leq C, 
\end{align*}
which completes the proof of the corollary.  
\end{proof}
\end{cor}
\bigskip
\subsection{Quantitative estimates of Hilbert modular cusp forms}\label{results2}
\noindent
In this section, we derive quantitative estimates of the Bergman kernel associated to classical Hilbert modular cusp forms, following the heat kernel technique from \cite{jk2}. For the rest of the article, we assume that the quaternion algebra $B$ is split at all finite places of $F$.
\begin{thm}\label{proofmainthm2}
Let the notation be as above, and let $\Gamma^1$ be of {\it{Type (1)}}. Then, for any $\uk=(k_1,\ldots,k_r)\in (2\mathbb{Z}_{>0}^{r})$, we have the following estimate
 \begin{align*}
\sup_{z\in X^1} \mathcal{B}_{X^1}^{\uk}(z) \leq\Bigg(36+ \frac{1}{\sinh^2(r_{X^1}/4)} \Bigg)^r \cdot  \prod_{j=1}^{r}k_j=O_{X^1}\Bigg(\prod_{j=1}^{r}k_j\Bigg).
 \end{align*}
 \end{thm}
 \begin{proof}
Using inequality \eqref{hkeqn} from section \ref{subsechkestimates}, for any $z=(z_1,\ldots,z_r)\in X^1$, we find
 \begin{align}\label{mainthm2proof1}
 \mathcal{B}_{X^1}^{\uk}(z) & \leq \sum_{\gamma=(\gamma_1,\ldots,\gamma_r)\in\Gamma^1}\,\prod_{j=1}^{r}\frac{8k_j^2\cdot e^{-2\rho_{\gamma_j, z_j}}}{\pi(k_j+1\slash 2)} \notag \\
& \leq \sum_{\gamma=(\gamma_1,\ldots,\gamma_r)\in\Gamma^1}\,\prod_{j=1}^{r} 4k_j\cdot e^{-2\rho_{\gamma_j, z_j}} \\
&\leq \prod_{j=1}^{r} 4k_j\cdot\sum_{\gamma=(\gamma_1,\ldots,\gamma_r)\in\Gamma^1} e^{-2\rho_{\gamma_j, z_j}}. \notag
 \end{align}
For any $1\leq j \leq r$, observe that the function $e^{-2\rho_{\gamma_j, z_j}}$ is a smooth, positive, real-valued, and decreasing function on $\mathbb{R}_{>0}$. For any $1\leq j\leq r$, 
observe that 
\begin{align*}
 \sum_{\gamma=(\gamma_1,\ldots,\gamma_r) \in \Gamma^1} e^{-2\rho_{\gamma_j, z_j}}  \leq \int_0^{\infty} e^{-2\rho} dN_{\Gamma}(z_j;\rho).
 \end{align*}
 As the function $e^{-2\rho}$ is a monotonically decreasing function for $\rho \in \mathbb{R}_{\geq 0},$ for any $1\leq j \leq r$, using inequality \eqref{jlineq}, we derive
 \begin{multline}\label{mainthm2proof2}
 \sum_{\gamma=(\gamma_1,\ldots,\gamma_r)\in\Gamma^1 } e^{-2\rho_{\gamma_j, z_j}} \leq \int_0^{3r_{X^1}\slash 4} e^{-2\rho} dN_{\Gamma}(z_j;\rho) 
 + \frac{e^{-3r_{X^1}\slash 2}\sinh(r_{X^1}/2) \sinh(3r_{X^1}/4)}{\sinh^2(r_{X^1}/4)}  \\ 
 + \frac{1}{2 \sinh^2(r_{X^1}/4)} \int_{3r_{X^1}\slash 4}^{\infty} e^{-2\rho} \sinh\big(\rho + r_{X^1}\slash 2\big) d\rho.
\end{multline}
Express the three terms in the RHS of the above inequality as $T_1 + T_2 + T_3$. From the definition of the injectivity radius $r_{X^1}$ in \eqref{injecttype1}, we have
\begin{align}\label{mainthm2proof3}
T_1 = 1.
\end{align}
Using the fact that $\sinh(u)$ is a monotone increasing function, and the inequality $\cosh(u) \leq e^u$ holds for all $u \geq 0$, we have 
\begin{align}\label{mainthm2proof4}
T_2 &\leq \frac{e^{-3r_{X^1}\slash 2}\sinh(r_{X^1}\slash 2)\sinh(r_{X^1})}{\sinh^2(r_{X^1}\slash4)} \notag \\
&= 8 e^{-3r_{X^1}\slash2} \cosh^2(r_{X^1}\slash4) \cosh(r_{X^1}\slash 2)
 \\ 
 &\leq 8 e^{-3r_{X^1}\slash 2} e^{r_{X^1}} \leq 8. \notag
\end{align}
Using the fact that $\sinh(u) \leq e^u/2$ for all $u \geq 0$, we derive the  estimate
\begin{align}\label{mainthm2proof5}
T_3 \leq \frac{e^{r_{X^1}\slash2}}{4 \sinh^2(r_{X^1}/4)}\int_{3r_{X^1}\slash4}^{\infty} e^{-\rho} d\rho
  =  \frac{e^{-r_{X^1}\slash{4}}}{4\sinh^2(r_{X^1}/4)}
 \leq  \frac{1}{4\sinh^2(r_{X^1}\slash4)}.
\end{align}
For any $1\leq j \leq d$, combining inequalities \eqref{mainthm2proof2}, \eqref{mainthm2proof3}, \eqref{mainthm2proof4}, and \eqref{mainthm2proof5}, we arrive at the following estimate
\begin{align}\label{mainthm2proof6}
 \sum_{\gamma \in \Gamma^1 } e^{-2\rho_{\gamma_j, z_j}}  \leq \Bigg(9+ \frac{1}{4\sinh^2(r_{X^1}\slash4)} \Bigg).
\end{align}
Combining estimates \eqref{mainthm2proof1} and \eqref{mainthm2proof6} completes the proof of the theorem.
 \end{proof}
\begin{rem}
The estimate that we derived for $\mathcal{B}_{X^1}^{k}(z)$ in Theorem \ref{proofmainthm2} depends only on the injectivity radius $r_{X^1}$, which is bounded from below, as $X^1$  is compact. Furthermore, following similar arguments as in \cite{anil2} (Lemma $6.4$ in section $6$), it is easy to see that the lower bound for $r_{X^1}$ remains stable in covers,  which implies that our estimate for $\mathcal{B}_{X^1}^{k}(z)$ is stable in covers of compact Hilbert modular varieties.  
\end{rem}
Let $\Gamma^1$ be of {\it{Type (2)}}, i.e., $\Gamma^1$ is a finite index subgroup of $\mathrm{PSL}_{2}(\mathcal{O}_{F})$ without elliptic fixed points. Recall that we assume that $\Gamma$ has no hyperbolic-elliptic elements. Both these two assumptions are only to ease the notation, and the case when  $\Gamma^1$ admits both elliptic and hyperbolic-elliptic elements can be easily tackled.

Furthermore, without loss of generality, let us assume that $\infty:=(\infty,\ldots,\infty)$ is the only cusp of $\Gamma^1$ with stabilizer $\Gamma^{1}_{\infty}$. This assumption is only to ease the notational complexity in the proofs of the next two results. 

Put $\Gamma_0^1 := \mathrm{PSL}_{2}(\mathcal{O}_{F})$, and let $\Gamma_{0,\infty}^{1}$ denote the stabilizer of the cusp $\infty$ in $\Gamma^{1}_{0}$. Let $\{\sigma_1,\ldots\sigma_d\}$ denote the set of embeddings of the number field $F$ in $\mathbb{R}$. So, we have
\begin{align*}
\Gamma^1_{\infty} \subset \Gamma^1_{0,\infty}= \\&\Bigg \{ \bigg(
\begin{matrix}
 \varepsilon_1 & \alpha_1\\
 0 & \varepsilon^{-1}_1
\end{matrix}\bigg),\dots,\bigg(\begin{matrix}
 \varepsilon_d & \alpha_d\\
 0 & \varepsilon^{-1}_d
 \end{matrix} \bigg)\Big |\,\varepsilon\in \mathcal{O}^{\times}_{F},  \sigma_j(\varepsilon)=\varepsilon_j;\, \alpha\in \mathcal{O}_{F}, \sigma_j(\alpha)=\alpha_j;\,
1\leq j\leq d \Bigg \}.
\end{align*}

The following computation is useful in proving estimate \eqref{thm2estimate2}. 

\begin{lem}\label{auxlemma}
With notation as above, for a fixed $\uk=(k_1,\ldots,k_d)\in (2\mathbb{Z}_{>0}^{d})$, a fixed $\varepsilon\in\mathcal{O}_{F}^{\times}$, and a fixed $z=(z_1=x_1+iy_1,\ldots,z_d=x_d+iy_d)\in\calh^d$, we have the following estimate
\begin{align}\label{auxlemmaeqn}
\sum_{\alpha \in \mathcal{O}_{F}} \prod_{j=1}^d \frac{\big({4\varepsilon^2_j y^2_j}\big)^{k_{j}}}{\bigg(\big((1- \varepsilon^2_j)x_j-\varepsilon_j\alpha_j\big)^2 + \big(1+\varepsilon^2_j\big)^2 y^2_j\bigg)^{k_j}} \leq \prod_{j=1}^{d} \frac{\Gamma(k_j- 1/2)}{\Gamma(k_j)}\cdot  \frac{\big(\varepsilon_{j}\big)^{2k_{j}-1}y_j} {\big((1+ \varepsilon^2_j)\slash2 \big)^{2k_j-1}}.
 \end{align}
\end{lem}
\begin{proof}
For a fixed, fixed $\varepsilon\in\mathcal{O}_{F}^{\times}$, and a fixed $z=(z_1=x_1+iy_1,\ldots,z_d=x_d+iy_d)\in\calh^d$, using the fact that $\mathcal{O}_{F}$ is isomorphic to $\mathbb{Z}^d$ as a lattice, we compute
\begin{align}\label{proofauxlemmaeqn1} 
&\sum_{\alpha \in \mathcal{O}_{F}} \prod_{j=1}^d \frac{\big({4\varepsilon^2_j y^2_j}\big)^{k_{j}}}{\bigg(\big((1- \varepsilon^2_j)x_j-\varepsilon_j\alpha_j\big)^2 + \big(1+\varepsilon^2_j\big)^2 y^2_j\bigg)^{k_j}} \notag\\
 &= \sum_{\alpha^{\prime} \in {\frac{1}{2} \mathcal{O}_{F}}} \prod_{j=1}^d \frac{\big(\varepsilon_{j}y_j\big)^{2k_j}}{\bigg(\bigg(\frac{(1- \varepsilon^2_j)x_j}{2} -\varepsilon_j\alpha^{\prime}\bigg)^2 + \frac{(1+\varepsilon^2_j)^2 y^2_j}{4}\bigg)^{k_{j}}} 
 \notag\\  
 & \leq \prod_{j=1}^d  \sum_{\alpha^{\prime} \in {\frac{1}{2} \mathcal{O}_{F}}}\frac{\big(\varepsilon_{j}y_j\big)^{2k_{j}} d\alpha_j^{\prime}}{\bigg({\bigg(\frac{(1- \varepsilon^2_j)x_j}{2} -\varepsilon_j\alpha_j^{\prime}\bigg)^2 + \frac{(1+\varepsilon^2_j)^2 y^2_j}{4}}\bigg)^{k_j}} \notag\\
& \leq \prod_{j=1}^d  \int_{-\infty} ^{\infty} \frac{\big(\varepsilon_{j}y_j\big)^{2k_{j}} d\alpha_j^{\prime}}{\bigg({\bigg(\frac{(1- \varepsilon^2_j)x_j}{2} -\varepsilon_j\alpha_j^{\prime}\bigg)^2 + \frac{(1+\varepsilon^2_j)^2 y^2_j}{4}}\bigg)^{k_j}}.
  \end{align}
For any $1\leq j\leq d$, put
\begin{align*}
 \mathcal{I }_j :=   \int_{-\infty} ^{\infty} \frac{\big(\varepsilon_{j}{y_j}\big)^{2k_j} d\alpha^{\prime}_j}{\bigg(\bigg(\frac{(1- \varepsilon^2_j)x_j}{2} -\varepsilon_j \alpha^{\prime}_j\bigg)^2 + \frac{(1+\varepsilon^2_j)^2 y^2_j}{4}\bigg)^{k_{j}}}.
\end{align*}
Substituting
\begin{align*}
\Theta_j =\frac{(1- \varepsilon^2_j)x_j}{2} -\varepsilon_j \alpha^{\prime}_j \implies d \Theta_j = - \varepsilon_jd\alpha^{\prime}_j,
\end{align*} 
we find that
\begin{align*}
\mathcal{ I}_j  =  2\int_0 ^{\infty}
 \frac{\big(\varepsilon_{j}\big)^{2k_{j}-1}\big(y_j\big)^{2k_j} d\Theta_j}   {\bigg( \Theta^2_j + \frac{(1+\varepsilon^2_j)^2 y^2_j}{4}\bigg)^{k_j}}.
\end{align*}
Now put 
\begin{align*}
\beta_j = \frac{2\Theta_j}{\big(1+ \varepsilon^2_j\big)y_j} \implies d \Theta_j = \frac{\big(1+ \varepsilon^2_j\big)y_j}{2} d \beta_j.
\end{align*} 
So, we arrive at
\begin{align*}
 \mathcal{I}_j \leq  \frac{\big(\varepsilon_{j}\big)^{2k_{j}-1}y_j}{\big((1+ \varepsilon^2_j)\slash2\big)^{2k_j-1}} \int_0 ^{\infty} \frac{d\beta_j}{\big( \beta^2_j+1 \big)^{k_j}}.
\end{align*}
Now from formula 3.251.2 from \cite{gr}, we have
\begin{align*}
 \int_0 ^{\infty} \frac{d\beta_j}{\left( \beta^2_j+1 \right)^{k_j}} = \frac{\sqrt{\pi}\Gamma(k_j-1/2)}{2\Gamma(k_j)} \leq \frac{\Gamma(k_j-1/2)}{\Gamma(k_j)},
\end{align*}
using which, we can conclude that 
\begin{align}\label{proofauxlemmaeqn2}
 \mathcal{I}_j \leq\frac{\Gamma(k_j- 1/2)}{\Gamma(k_j)}\cdot  \frac{\big(\varepsilon_{j}\big)^{2k_{j}-1}y_j} {\big((1+ \varepsilon^2_j)\slash2 \big)^{2k_j-1}}.
\end{align}
Combining the inequalities \eqref{proofauxlemmaeqn1} and \eqref{proofauxlemmaeqn2} completes the proof of the lemma.
\end{proof}
\begin{thm}\label{proofmainthm2.2}
Let the notation be as above, and let $\Gamma^1$ be of {\it{Type (2)}}. Then, for any $\uk=(k_1,\ldots,k_d)\in (2\mathbb{Z}_{>0}^{d})$, we have the following estimate
 \begin{align}\label{proofmainthm2.2eqn}
\sup_{z\in X^1} \mathcal{B}_{X^1}^{\uk}(z)=O_{X^1}\Bigg(\prod_{j=1}^{d} k_j^{3\slash 2}\Bigg).
 \end{align}
 \end{thm}
\begin{proof}
For any $\uk=(k_1,\ldots,k_d)\in (2\mathbb{Z}_{>0}^{d})$ and $z=(z_1,\ldots,z_d)\in X^1$, from inequality \eqref{hkeqn0} from section \ref{subsechkestimates}, we have
\begin{multline}\label{proofmainthm2.2eqn1}
\mathcal{B}_{X^1}^{\uk}(z)\leq \sum_{\gamma=(\gamma_1,\ldots,\gamma_d)\in\Gamma^1}\,\prod_{j=1}^{d}K_{\calh}^{k_j}(t;z_j,\gamma_jz_j) \\
= \sum_{\gamma=(\gamma_1,\ldots,\gamma_d)\in\Gamma^1\backslash\Gamma^{1}_{\infty}}\,\prod_{j=1}^{d}K_{\calh}^{k_j}(t;z_j,\gamma_jz_j) +\sum_{\gamma=(\gamma_1,\ldots,\gamma_d)\in\Gamma^1_{\infty}}\,\prod_{j=1}^{d}K_{\calh}^{k_j}(t;z_j,\gamma_jz_j) .
\end{multline}
Following similar arguments that went into the proof of Theorem \ref{proofmainthm2}, we have the following estimate for the first term on the right-hand side of the above equality
\begin{align}\label{proofmainthm2.2eqn2}
\sum_{\gamma=(\gamma_1,\ldots,\gamma_d)\in\Gamma^1\backslash\Gamma^1_{\infty}}\,\prod_{j=1}^{d} K_{\calh}^{k_j}(t;z_j,\gamma_jz_j)=\bigg(36+\frac{1}{\sinh^{2}(r_{X^1}\slash 4)}\bigg)^{d} \cdot \prod_{j=1}^{d}k_{j}=O_{X^1}\bigg(\prod_{j=1}^{d}k_{j}\bigg),
\end{align}
where $r_{X^1}$ is as defined in equation \eqref{injecttype2}.

From inequality  (2.5) of \cite{abms}, for any $1\leq j \leq d$, we have the following estimate 
 \begin{align*}
K_{\calh}^{k_j}(t;z_j,\gamma_jz_j) \leq \frac{2k_j^2}{\pi(k_j+1\slash 2)}\cdot \frac{e^{-{\rho_{\gamma_j, z_j}}}}{\cosh^{2k_j}({\rho_{\gamma_j, z_j}}/{2})},
 \end{align*}
 using which, we get
\begin{multline}\label{proofmainthm2.2eqn3}
 \sum_{\gamma=(\gamma_1,\ldots,\gamma_d) \in \Gamma^1_{\infty} }\, \prod_{j=1}^{d}K_{\calh}^{k_j}(t;z_j,\gamma_jz_j) \leq\sum_{\gamma =(\gamma_1,\ldots,\gamma_d)\in \Gamma^1_{\infty}}\,
 \prod_{j=1}^d\frac{2k_{j}^2}{\pi(k_j+1\slash2)} \cdot \frac{e^{-{\rho_{\gamma_j, z_j}}}}{\cosh^{2k_j}(\rho_{\gamma_j, z_j}\slash 2)} \\ 
 = \bigg(\prod_{j=1}^{d}\frac{2k_{j}^2}{\pi(k_j+1\slash2)}\bigg)\cdot\sum_{\gamma =(\gamma_1,\ldots,\gamma_d)\in \Gamma^1_{\infty}} \,\prod_{j=1}^d \frac{e^{-{\rho_{\gamma_j, z_j}}}}{\cosh^{2k_j}(\rho_{\gamma_j, z_j}\slash 2)}.
 \end{multline}
 Recall that any $\gamma\in\Gamma^1_{\infty}$ is of the form
 \begin{align*}
 \gamma=\bigg(\begin{matrix}
 \varepsilon & \alpha\\
 0 & \varepsilon^{-1}
\end{matrix} \bigg)\in\Gamma^1_{\infty},\,\,\alpha\in\mathcal{O}_{F}, \,\,\mathrm{and} \,\,\varepsilon\in\mathcal{O}_{F}^{\times},
 \end{align*}
 which implies that for $z\in X^1$ and $1\leq j\leq d$, we have
 \begin{align*}
 \gamma_jz_j=\varepsilon_j^{2}z_j +\varepsilon_j \alpha_j.
 \end{align*}
For any $z,w\in\calh$, using the formula 
\begin{align*}
 \cosh^2 \big( \mathrm{d}_{\calh}(z, w) \slash 2\big) = \frac{{\big| z- w \big|}^2}{4\mathrm{Im} (z)\mathrm{Im} (w)} +1,
\end{align*}
for any $1\leq j\leq d$ and $z=(z_1,\ldots,z_d)\in X^1$, we compute,
\begin{align*}
 \cosh^{2}(\rho_{\gamma_j, z_j}\slash2) =&\cosh^{2}\big(d_{\calh}(z_j,\gamma_j z_j)\slash 2\big)= \frac{\big|z_j- \gamma_j z_j \big|^2}{4 \mathrm{Im}(z_j) \mathrm{Im} (\gamma_j z_j)} +1   \\  =& \frac{\big((1- \varepsilon^2_j)x_j-\varepsilon_j \alpha_j\big)^2 + (1+\varepsilon^2_j)^2 y^2_j}{4\varepsilon^2_j y^2_j}.
\end{align*}
Using the above equation, and the fact that for all $1\leq j\leq d$,  we have $e^{-{\rho_{\gamma_j, z_j}}} \leq 1$. So using Lemma \ref{auxlemma}, we arrive at the following inequality
\begin{align}\label{proofmainthm2.2eqn4}
 \sum_{\gamma =(\gamma_1,\ldots,\gamma_d)\in \Gamma^1_{\infty}} \,\prod_{j=1}^d \frac{e^{-{\rho_{\gamma_j, z_j}}}}{\cosh^{2k_j}({\rho_{\gamma_j, z_j}}/{2})} & \leq \sum_{\varepsilon\in \mathcal{O}^{\times}_{F}} \sum_{\alpha \in \mathcal{O}_{F}}\prod_{j=1}^d \frac{\big(4\varepsilon^2_j y^2_j\big)^{k_{j}}}{\bigg(\big((1- \varepsilon^2_j)^2-\varepsilon_j \alpha_j\big)^2 + \big(1+\varepsilon^2_j\big)^2 y^2_j\bigg)^{k_{j}}}\notag
 \\ & \leq \bigg( \prod_{j=1}^d  \frac{\Gamma(k_j- 1\slash2)}{\Gamma(k_j)} \bigg)\cdot\sum_{\varepsilon \in \mathcal{O}^{\times}_{K}}\prod_{j=1}^d  \frac{\big(\varepsilon_j\big)^{2k_{j}-1}y_j}{\big((1+ \varepsilon^2_j)\slash2 \big)^{2k_j-1}}.
\end{align}
Now  for any $1\leq j\leq d$, using the simple fact
\begin{align*}
 \frac{2}{1+\varepsilon_j^2} \leq \frac{1}{\varepsilon_j}\,\,\mathrm{and}\,\, \prod_{j=1}^d \varepsilon_j=\pm1,
 \end{align*}
and the fact that the group of units $\mathcal{O}_{F}^{\times}$ is a sub-lattice of $\mathbb{Z}^{d}$, we derive
\begin{align}\label{proofmainthm2.2eqn5}
  \sum_{\varepsilon \in \mathcal{O}^{\times}_{F}}\prod_{j=1}^d\Bigg| \frac{\big(\varepsilon_j\big)^{2k_{j}-1}y_j}{\big((1+ \varepsilon^2_j)\slash2 \big)^{2k_j-1}} \Bigg|& \leq 
  \sum_{\varepsilon \in \mathcal{O}^{\times}_{F}}\prod_{j=1}^{d} \Bigg|\frac{\big(\varepsilon_j\big)^{2k_{j}-1}} {(\varepsilon_j)^{2k_j-2}}\Bigg|\cdot\frac{y_j}{\big((1+ \varepsilon^2_j)\slash2 \big)}
  =\sum_{\varepsilon \in \mathcal{O}^{\times}_{F}}  \prod_{j=1}^d \frac{2y_j } {\big(1+ \varepsilon^2_j \big)} \notag\\
 & \leq \prod_{j=1}^d   \sum_{\varepsilon \in \mathcal{O}^{\times}_{F}}\frac{2y_j } {\big(1+ \varepsilon^2_j \big)}\leq
 \prod_{j=1}^d \int_{-{\infty}}^{\infty} \frac{2y_j\mathrm{d} \varepsilon_j}{\big(1+ \varepsilon^2_j \big)}= \prod_{j=1}^d 2\pi y_j. 
\end{align}
Combining inequalities \eqref{proofmainthm2.2eqn3}, \eqref{proofmainthm2.2eqn4}, and \eqref{proofmainthm2.2eqn5}, we arrive at the following inequality
\begin{align*}
 \sum_{\gamma=(\gamma_1,\ldots,\gamma_d) \in \Gamma^1_{\infty} } \,\prod_{j=1}^{d} K_{\calh}^{k_j}(t;z_j,\gamma_jz_j) \leq\prod_{j=1}^{d} \Bigg(4 y_j\cdot \frac{\Gamma(k_j- 1/2)}{\Gamma(k_j)} \cdot  \frac{k_{j}^2}{(k_j+1\slash2)}\Bigg).
\end{align*}
Now adapting the same arguments as in p. 12 in section $5$ of \cite{jk2}, for any $1\leq j\leq d$, we get
\begin{align*}
y_j\cdot \frac{\Gamma(k_j- 1/2)}{\Gamma(k_j)} =O(\sqrt{k_j}), 
\end{align*}
which implies that
\begin{align}\label{proofmainthm2.2eqn6}
\sum_{\gamma=(\gamma_1,\ldots,\gamma_d) \in \Gamma^1_{\infty} }\, \prod_{j=1}^{d}K_{\calh}^{k_j}(t;z_j,\gamma_jz_j) \leq\prod_{j=1}^{d} \Bigg(4  y_j\cdot \frac{\Gamma(k_j- 1/2)}{\Gamma(k_j)} \cdot
 \frac{k_{j}^2}{(k_j+1\slash2)}\Bigg)=O\big(\prod_{j=1}^{d}k_{j}^{3\slash 2}\big).
 \end{align}
 The proof of the theorem follows by combining the estimates \eqref{proofmainthm2.2eqn2} and \eqref{proofmainthm2.2eqn6}.
 \end{proof}
\begin{rem}
The estimate that we derived for $\mathcal{B}_{X^1}^{\uk}(z)$ in Theorem \ref{proofmainthm2.2} depends only the injectivity radius $r_{X^1}$, which is bounded from below by the injectivity radius of $X_{0}=\mathrm{PSL}_{2}(\mathcal{O}_{F})\backslash\calh^d$. This implies that 
\begin{align*}
\sup_{z\in X^1} \mathcal{B}_{X^1}^{\uk}(z)=O\Bigg(\prod_{j=1}^{d}k_j^{3\slash 2}\Bigg),
\end{align*}
i.e., the implied constant in the estimate on the right hand side of equation  \eqref{proofmainthm2.2eqn} is a universal constant depending only on the number field $F$. 
\end{rem}

{\bf{Acknowledgements}}.
Both the authors would like to thank the Referee for his comments, which have greatly helped in improving the quality  of the exposition. Both the authors would also like to thank the  Mathematics section of ICTP, Trieste for their support and hospitality, and for providing a congenial atmosphere where part of this article was realized.  

 The first author acknowledges the support of INSPIRE research grant DST/INSPIRE/04/2015/\\002263. The second author was partially supported by SERB grant EMR/2016/000840.

{}
\end{document}